\documentclass[11pt]{amsart}

\usepackage{enumerate, amsmath, amsthm, amsfonts, amssymb, xy,  mathrsfs, graphicx, paralist, lscape}
\usepackage[usenames, dvipsnames]{color}
\usepackage{fullpage}
\usepackage{comment}
\usepackage{tikz}
\usetikzlibrary{matrix,arrows}

\numberwithin{equation}{section}
\newtheorem{theorem}[equation]{Theorem}
\newtheorem*{thm}{Theorem}
\newtheorem*{mainthm}{Main Theorem}
\newtheorem{proposition}[equation]{Proposition}
\newtheorem{lemma}[equation]{Lemma}
\newtheorem{corollary}[equation]{Corollary}

\newtheorem*{coro}{Corollary}
\newtheorem*{questn}{Question}

\newtheorem{quest}[equation]{Question}

\theoremstyle{definition}
\newtheorem{rmk}[equation]{Remark}
\newenvironment{remark}[1][]{\begin{rmk}[#1] \pushQED{\qed}}{\popQED \end{rmk}}
\newtheorem{rmks}[equation]{Remarks}

\newtheorem{eg}[equation]{Example}
\newenvironment{example}[1][]{\begin{eg}[#1] \pushQED{\qed}}{\popQED \end{eg}}
\newtheorem{defn}[equation]{Definition}
\newenvironment{definition}[1][]{\begin{defn}[#1]\pushQED{\qed}}{\popQED \end{defn}}
\newtheorem{altdefn}[equation]{Alternative Definition}
\newenvironment{altdefinition}[1][]{\begin{altdefn}[#1]\pushQED{\qed}}{\popQED \end{altdefn}}
\newtheorem{ques}[equation]{Question}
\newenvironment{question}[1][]{\begin{ques}[#1]\pushQED{\qed}}{\popQED \end{ques}}
\newtheorem{notn}[equation]{Notation}

\usepackage[colorlinks=true, pdfstartview=FitV, linkcolor=blue, citecolor=blue, urlcolor=blue]{hyperref}

\renewcommand{\phi}{\varphi}
\renewcommand{\emptyset}{\varnothing}

\renewcommand{\tilde}[1]{\widetilde{#1}}

\newcommand{\hgt}{\mbox{ht}}

\makeatletter
\def\Ddots{\mathinner{\mkern1mu\raise\p@
\vbox{\kern7\p@\hbox{.}}\mkern2mu
\raise4\p@\hbox{.}\mkern2mu\raise7\p@\hbox{.}\mkern1mu}}
\makeatother

\DeclareMathOperator{\ass}{Ass}

\usetikzlibrary{calc}
\include{pipe-macros-stumpf}

\title{Geometric vertex decomposition and liaison}

\author{Patricia Klein}
\author{Jenna Rajchgot}
\address{University of Minnesota, School of Mathematics, Minneapolis, MN, USA.}
\email[Patricia Klein]{klein847@umn.edu}
\address{McMaster University, Department of Mathematics and Statistics, Hamilton, ON, Canada}
\email[Jenna Rajchgot]{rajchgot@math.mcmaster.ca}

\begin{document}

\maketitle

\begin{abstract}
 Geometric vertex decomposition and liaison are two frameworks that have been used to produce similar results about similar families of algebraic varieties. 
 In this paper, we establish an explicit connection between these   approaches.
 In particular, we show that each geometrically vertex decomposable ideal is linked by a sequence of elementary $G$-biliaisons of height 1 to an ideal of indeterminates and, conversely, that every $G$-biliaison of a certain type gives rise to a geometric vertex decomposition. As a consequence, we can immediately conclude that several well-known  families of ideals are  glicci, including Schubert determinantal ideals, defining ideals of varieties of complexes, and defining ideals of graded lower bound cluster algebras. 
\end{abstract}

\section{Introduction}

Determinantal ideals and their generalizations have been explored extensively both in the context of commutative algebra and also in the study of Schubert varieties in flag varieties. This overlap is to be expected because, for example, 
\begin{itemize}
\item each ideal generated by the $k\times k$ minors of a generic matrix is the defining ideal of an open patch of a Schubert variety in a Grassmannian;
\item each one-sided ladder determinantal ideal is a Schubert determinantal ideal for a vexillary (i.e., 2143-avoiding) permutation (see eg. \cite{KMY09});
\item each two sided mixed ladder determinantal ideal is a type $A$ Kazhdan-Lusztig ideal (see eg. \cite{EFRW});
\item each ideal generated by the $k\times k$ minors of a generic symmetric matrix is the  defining ideal of an open patch of a Schubert variety in a Lagrangian Grassmannian;
\item each defining ideal of a variety of complexes is a type $A$ Kazhdan-Lusztig ideal, up to some extra indeterminate generators (see eg. \cite[Ch. 17]{MillerSturmfels}).
\end{itemize}

While similar results on the above-mentioned families of ideals appear in the Schubert variety and commutative algebra literatures, it is often different techniques that are used to obtain them. 

For example, in \cite{KMY09}, A. Knutson, E. Miller, and A. Yong introduced \emph{geometric vertex decomposition}, a degeneration technique, and used this to study Gr\"obner geometry of Schubert determinantal ideals for vexillary permutations. 
See Section \ref{sect:GVD} for background on geometric vertex decomposition. 
Independently, \emph{liaison-theoretic} methods were used by E. Gorla in \cite{Gor07} and E. Gorla, J. Migliore, and U. Nagel in \cite{GMN13} to obtain Gr\"obner bases for  various classes of ladder determinantal ideals (including one sided ladder determinantal ideals,
also known as Schubert determinantal ideals for vexillary permutations). 
Roughly speaking, liaison is a theory that aims to transfer information from one subscheme of projective space to another in cases when their union is sufficiently nice. See Section \ref{sect:liaison} for background on liaison. 

In this paper, we establish an explicit connection between  \emph{geometric vertex decomposition} and \emph{liaison}, and we study implications of this connection. We have three main goals, which we now outline.

\subsubsection*{First goal} The first goal of this paper is to show that it is no coincidence that geometric vertex decomposition and liaison can be used to obtain similar results for similar classes of ideals. Indeed, we prove the following explicit connection between the two techniques:

\begin{mainthm}\label{thm:main1}
Under mild hypotheses, every geometric vertex decomposition gives rise to an elementary $G$-biliason of height $1$.  Every sufficiently ``nice" elementary $G$-biliaison of height $1$ gives rise to a geometric vertex decomposition.
\end{mainthm}

The first half of this theorem is stated precisely and proved as Corollary \ref{cor:gvdToLia}.  The second half is stated precisely and proved as Theorem \ref{thm:linkToGVD}.  

\subsubsection*{Second goal.} The second motivation for our work comes from a long-standing open question in liaison theory, which asks whether subschemes of $\mathbb{P}^n$ are arithmetically Cohen--Macaulay if and only if they are in the \emph{Gorenstein liaison class of a complete intersection}  (often referred to as \emph{glicci}, shorthand introduced in \cite{KM+01}).  It is a standard homological argument that every glicci subscheme of $\mathbb{P}^n$ is arithmetically Cohen--Macaulay.  Hence, the question may be phrased as follows:

\begin{questn}\label{motivation2}\cite[Question 1.6]{KM+01}
Is every arithmetically Cohen--Macaulay subscheme of $\mathbb{P}^n$ glicci?
\end{questn}

\noindent For more background on why this question emerges naturally from the history of liaison and for a summary of partial results already in the literature, see Section \ref{sect:liaison}.

By combining our main theorem with some straightforward consequences of geometric vertex decomposition, we arrive at the following, which is stated precisely as Corollary \ref{cor:AutomaticallyGlicci}:

\begin{coro}\label{cor:mainCor}
Let $I$ be a homogenous ideal in a polynomial ring. 
If the Lex-initial ideal of $I$ is the Stanley--Reisner ideal of a vertex decomposable simplicial complex and the vertex decomposition is compatible with the order of the variables, then $I$ is glicci. 
\end{coro}

From this corollary, one can quickly deduce that certain well-known classes of varieties are glicci. We discuss three such classes in Section \ref{sect:applications}: matrix Schubert varieties, varieties of complexes, and varieties of graded lower bound cluster algebras.

Using the first half of our main theorem, we recover a result of U. Nagel and T. R\"omer from \cite{NR08}, namely that the Stanley--Reisner ideal of a vertex decomposable simplicial complex is glicci. 
In fact, U. Nagel and T. R\"omer showed, more generally, that the Stanley--Reisner ideal of a \emph{weakly vertex decomposable} simplicial complex is glicci  \cite[Theorem 3.3]{NR08}.  
Taking this as motivation, we define the class of \emph{weakly geometrically vertex decomposable} ideals (Definition \ref{def:weaklyGeometricallyVertexDec}), which includes both the geometrically vertex decomposable ideals and the Stanley--Reisner ideals of weakly vertex decomposable complexes. We show the following, labeled as Corollary \ref{weakGVDimpliesgliggi} in the main body of the paper:

\begin{thm}
Weakly geometrically vertex decomposable ideals are glicci.
\end{thm}

\subsubsection*{Third goal.} 
In  \cite[Lemma 1.12]{GMN13}, it is shown that one can use liaison to compare Hilbert functions when the degrees of the isomorphisms of the $G$-biliaisons involved in an inductive argument are known.  This approach is employed in many of the determinantal cases treated in the literature (\cite{Gor07, Gor08, GMN13, FK20}). It is worth noticing that the isomorphisms employed in these papers all have a similar form. 
We explain via Theorem \ref{thm:onestep} why this similarity is not a coincidence but, rather, is to be expected.  In that theorem, we associate an explicit isomorphism of degree $1$ to a geometric vertex decomposition.

In addition to the expository work of describing a unifying structure underlying examples already in the literature, Theorem \ref{thm:onestep} also provides a candidate isomorphism in the style of $G$-biliaison that, in good cases, allows one to use the framework of \cite{GMN13} to prove that a conjectured Gr\"obner basis is, indeed, a Gr\"obner basis.  Some consequences of Theorem \ref{thm:onestep} on Gr\"obner bases and degenerations appear in Subsection \ref{sect:GBapplications}.

\subsection*{The structure of the paper}

In Section \ref{sect:GVD}, we review definitions and key lemmas from \cite{KMY09} on geometric vertex decomposition in the unmixed case and record some additional observations about the structure of a geometrically vertex decomposable ideal.  In Section \ref{sect:liaison}, we briefly review background material on Gorenstein liaison.  In Sections \ref{sect:GVDGlicci} and \ref{sect:glicciGVD}, we provide a proof of our main theorem (stated above), and related results and examples. In Section \ref{sect:applications}, we prove that certain well-known classes of combinatorially-defined ideals are glicci, via the material in Section \ref{sect:GVDGlicci}.  Finally, we devote Section \ref{sect:nonpure} to the not necessarily unmixed case, which we relate to vertex decomposition in the not necessarily pure case.

\subsection*{Notational conventions} Throughout the paper, we let $\kappa$ be a field, which can be chosen arbitrarily except in Sections \ref{sect:GVDGlicci}, \ref{sect:applications}, and \ref{sect:nonpure}, where we require that $\kappa$ be infinite.

\subsection*{Acknowledgements}
We thank Sergio Da Silva, Elisa Gorla, Kuei-Nuan Lin, Yi-Huang Shen, Adam Van Tuyl, and Anna Weigandt for helpful conversations. We are also grateful to the anonymous referee for a very careful reading of the paper and for helpful feedback. Part of this work was completed at the Banff International Research Station (BIRS) during the Women in Commutative Algebra workshop in October 2019. We are grateful for the hospitality of the Banff Centre.  The second author was partially supported by NSERC grant RGPIN-2017-05732.

\section{Geometric vertex decomposition}\label{sect:GVD}

In this section we discuss geometric vertex decomposition, introduced by A. Knutson, E. Miller, and A. Yong in \cite{KMY09}. In the  first subsection, we recall the basics of vertex decomposition of simplicial complexes and Stanley--Reisner ideals. In the second subsection, we move beyond the monomial-ideal case and recall the basics of \emph{geometric} vertex decomposition from \cite{KMY09}. In the third subsection, we define and study \emph{geometrically vertex decomposable ideals}. Although the material in this last subsection is not known to the authors to be explicitly in the literature,
the results that appear will not be surprising to experts. 

\subsection{Vertex decomposition and Stanley--Reisner ideals}\label{sect:vertDecomp}
Let $\Delta$ be a simplicial complex on vertex set $[n] = \{1,2,\dots, n\}$ (without an insistence that every $v\in [n]$ necessarily be a face of $\Delta$). 
Given a vertex $v\in \Delta$, define the following three subcomplexes:
\begin{itemize}
\item the \textbf{star} of $v$ is the set $\text{star}_\Delta(v):= \{F\in \Delta\mid F\cup \{v\}\in \Delta\}$; 
\item the \textbf{link} of $v$ is the set $\text{lk}_{\Delta}(v) := \{F\in \Delta\mid F\cup\{v\}\in \Delta, F\cap \{v\}=\emptyset\}$;
\item the \textbf{deletion} of $v$ is the set $\text{del}_{\Delta}(v) := \{F\in \Delta\mid F\cap \{v\} = \emptyset\}.$ 
\end{itemize}
Recall that the \textbf{cone} from $v$ on a simplicial complex $\Delta$ is the smallest simplicial complex that contains the set $\{F\cup \{v\}\mid F\in \Delta\}$. 
Then $\text{star}_\Delta(v)$ is the cone from $v$ on $\text{lk}_\Delta(v)$  and 
\begin{equation}\label{eq:vertexDecomp}
\Delta = \text{star}_{\Delta}(v)\cup \text{del}_\Delta(v). 
\end{equation}

The decomposition of $\Delta$ in \eqref{eq:vertexDecomp} is called a \textbf{vertex decomposition}.

A simplicial complex is called \textbf{pure} if all of its facets (i.e., maximal faces) are of the same dimension. 
A simplicial complex $\Delta$ is \textbf{vertex decomposable} if it is pure and if $\Delta = \emptyset$, or $\Delta$ is a simplex, or there is a vertex $v\in \Delta$ such that $\text{lk}_\Delta(v)$ and $\text{del}_{\Delta}(v)$ are vertex decomposable.

Given a simplicial complex $\Delta$ on vertex set $[n]$, one defines the \textbf{Stanley--Reisner ideal} $I_{\Delta}\subseteq \kappa[x_1,\dots, x_n]$ associated to $\Delta$ as  $I_\Delta:= \langle \textbf{x}_F\mid F\subseteq [n], F\notin \Delta\rangle$, where $\textbf{x}_F := \prod_{i\in F}x_i$.  
The association $\Delta\mapsto I_\Delta$ determines a bijection between simplicial complexes on $[n]$ and squarefree monomial ideals in $\kappa[x_1,\dots, x_n]$. We write $\Delta(I)$ for the simplicial complex associated to a squarefree monomial ideal $I$. 

Notice that if $\Delta = \Delta_1\cup \Delta_2$ is a union of simplicial complexes on $[n]$, then $F$ is a non-face of $\Delta$ if and only if it is a non-face of both $\Delta_1$ and $\Delta_2$. Thus, $I_\Delta = I_{\Delta_1}\cap I_{\Delta_2}$. In particular, if $v$ is a vertex of $\Delta$, we may decompose $\Delta$ as in \eqref{eq:vertexDecomp} to get
\[
I_\Delta = I_{\text{star}_\Delta(v)}\cap I_{\text{del}_\Delta(v)}.
\]

The following is immediate from the definitions. We record it as a lemma for easy reference.

\begin{lemma}\label{lem:starDelLink}
Let $v\in [n]$ be a vertex of $\Delta$. Write $I_{\Delta} = \langle x_v^{d_i}q_i\mid 1\leq i\leq m\rangle$ where $q_i$ is a squarefree monomial that is not divisible by $x_v$ and $d_i = 0$ or $1$. Then
\[
I_{\text{star}_\Delta(v)} = \langle q_i \mid 1\leq i\leq m\rangle,\quad I_{\text{lk}_\Delta(v)} = I_{\text{star}_\Delta(v)}+\langle x_v\rangle, \quad
I_{\text{del}_\Delta(v)} = \langle q_i \mid d_i = 0\rangle+\langle x_v\rangle.
\]
\end{lemma}

\subsection{Geometric vertex decomposition}
In this subsection, we discuss \emph{geometric vertex decomposition}, introduced by A. Knutson, E. Miller, and A. Yong in \cite{KMY09}. 

Let $R = \kappa[x_1,\dots, x_n]$ be a polynomial ring in $n$ indeterminates and let $y = x_j$ for some $1 \leq j \leq n$.  Define the \textbf{initial $y$-form} $\text{in}_y f$ of a polynomial $f\in R$ to be the sum of all terms of $f$ having the highest power of $y$. That is, if $f = \sum_{i = 0}^n \alpha_i y^i$, where each $\alpha_i\in \kappa[x_1,\dots,\widehat{y},\dots, x_n]$ and $\alpha_n\neq 0$, define $\text{in}_y f :=\alpha_n y^n$, which is usually not a monomial. Given an ideal $J\subseteq R$, define $\text{in}_y J$ to be the ideal generated by the initial $y$-forms of the elements of $J$, that is, $\text{in}_y J: = \langle \text{in}_y f \mid f \in J \rangle.$ 
We say that a monomial order $<$ on $R$ is \textbf{$y$-compatible} if it satisfies
$\text{in}_< f = \text{in}_<(\text{in}_y f)$ 
for every $f\in R$. In this case, one has $\text{in}_<(\text{in}_y J) = \text{in}_< J$ for any ideal $J\subseteq R$.

Let $I\subseteq R$ be an ideal and $<$ a $y$-compatible monomial order.  
With respect to $<$, let $\mathcal{G} = \{y^{d_i}q_i+r_i\mid 1\leq i\leq m\}$ be a Gr\"obner basis of $I$ where $y$ does not divide any $q_i$ and $\text{in}_{y}(y^{d_i}q_i+r_i) = y^{d_i}q_i$. One easily checks that the ideal $\text{in}_y I$ is generated by $\text{in}_y \mathcal{G}:= \{y^{d_i}q_i\mid 1\leq i\leq m\}$. That is, $\text{in}_{y} I = \langle y^{d_i}q_i\mid 1\leq i\leq m\rangle$. 


\begin{definition}\label{def:gvdKMS}\cite[Section 2.1]{KMY09}
Define $C_{y,I} := \langle q_i\mid 1\leq i\leq m\rangle$ and $
N_{y,I} := \langle q_i\mid d_i = 0\rangle.$ 
When $\text{in}_y I = C_{y,I} \cap (N_{y,I}+\langle y \rangle)$, this decomposition is called a \textbf{geometric vertex decomposition of $I$ with respect to $y$}.  
\end{definition}

The ideals $C_{y,I}$ and $N_{y,I}$ do not depend on the choice of Gr\"obner basis and, in particular, do not depend on the choice of $y$-compatible term order $<$.  
This follows from the facts that $C_{y,I} = (\text{in}_y I :y^\infty)$ by \cite[Theorem 2.1(d)]{KMY09} and that $N_{y,I}+\langle y\rangle = \text{in}_y I+\langle y\rangle$ by \cite[Theorem 2.1 (a)]{KMY09}, together with the observation that $y$ does not appear in the generators of $N_{y,I}$ given in its definition. 

 We say that a geometric vertex decomposition is \textbf{degenerate} if $\sqrt{C_{y,I}} = \sqrt{N_{y,I}}$ or if $C_{y,I} = \langle 1 \rangle$ and \textbf{nondegenerate} otherwise. As we will see through Lemma \ref{lem:form}, if $C_{y,I} = \langle 1 \rangle$, then some polynomial whose initial $y$-form is a unit multiple of $y$ is an element of $I$, in which case $R/I \cong R/(N_{y,I}+\langle y \rangle)$.   If  $\sqrt{C_{y,I}} = \sqrt{N_{y,I}}$, then $\sqrt{\text{in}_y I} = \sqrt{C_{y,I}} \cap \sqrt{N_{y,I}+\langle y \rangle} = \sqrt{C_{y,I}}$, in which case $\text{in}_y I$, $C_{y,I}$, and $N_{y,I}$ all determine the same variety.  In both of these cases, we may often prefer to study $N_{y,I}$ in the smaller polynomial ring that omits $y$.  This is especially true when $I$ is radical for the following reason:
 
 \begin{proposition}\label{prop:degen-rad}
 If $I$ is radical and has a degenerate geometric vertex decomposition $\text{in}_y I = C_{y,I} \cap (N_{y,I}+\langle y \rangle)$ with $\sqrt{N_{y,I}} = \sqrt{C_{y,I}}$, then the reduced Gr\"obner basis of $I$ does not involve $y$ and $I = \text{in}_y I = C_{y,I} = N_{y,I}$.
 \end{proposition}
 \begin{proof}
 Throughout this argument, we will refer to the reduced Gr\"obner basis of $I$ as the Gr\"obner generators of $I$ and the generators of $N_{y,I}$ obtained in Definition \ref{def:gvdKMS} from the reduced Gr\"obner basis of $I$ as the Gr\"obner generators of $N_{y,I}$.
 
 We claim first that $N_{y,I}$ must also be radical.  Fix some $g^t \in N_{y,I}$ for $t \geq 1$.  Because $N_{y,I}$ has a generating set that does not involve $y$, we may assume without loss of generality that $g$ does not involve $y$.  Because $g^t \in N_{y,I} \subseteq I = \sqrt{I}$, we also have $g \in I$, and so $g$ must have a Gr\"obner reduction by elements of the reduced Gr\"obner basis of $I$.  Because $g$ does not involve $y$, this reduction must use only those Gr\"obner generators that do not involve $y$, which are exactly the Gr\"obner generators of $N_{y,I}$, and so $g \in N_{y,I}$.  
 
Hence, $\sqrt{C_{y,I}} = \sqrt{N_{y,I}} = N_{y,I} \subseteq C_{y,I}$, and so $N_{y,I} = C_{y,I}$.  Suppose now that the reduced Gr\"obner basis of $I$ has some element of the form $y^dq+r$ for $d>0$.  Then $q \in C_{y,I} = N_{y,I}$, and so the lead term of $q$ must be divisible by the lead term of one of the Gr\"obner generators of $N_{y,I}$.  But any such generator is also an element of the reduced Gr\"obner basis of $I$ and so cannot divide the lead term of $q$ since then it would divide the lead term of $y^dq+r$.  Hence, the reduced Gr\"obner basis of $I$ has no term involving $y$, from which it follows that $I = \text{in}_y I = C_{y,I} = N_{y,I}$.
 \end{proof}

\begin{remark}\label{rem:GVDsimplicial}
Let $\Delta$ be a simplicial complex on vertex set $[n]$, and let $v$ be a vertex of $\Delta$. 
The geometric vertex decomposition of $I_\Delta\subseteq R$ with respect to variable $x_v$ agrees with the decomposition 
\[
I_{\Delta} = I_{\text{star}_\Delta(v)}\cap I_{\text{del}_\Delta(v)}.
\]
Indeed, $\text{in}_{x_v}I_\Delta = I_\Delta$, $I_{\text{star}_\Delta(v)} = C_{x_v,I_\Delta}$, and $I_{\text{del}_\Delta(v)} = N_{x_v,I_{\Delta}}+\langle x_v\rangle$ (see the end of Section \ref{sect:vertDecomp}). 
Observe that since $v\in \Delta$, we have $C_{x_v,I_\Delta} \neq \langle 1\rangle$. Thus, the geometric vertex decomposition is degenerate if and only if $\Delta$ is a cone from $v$ on $\text{lk}_\Delta(v)$.
\end{remark}

If an ideal $I\subseteq R$ has a generating set $\mathcal{G}$ in which $y^2$ does not divide any term of $g$ for any $g \in \mathcal{G}$, then we say that $I$ is \textbf{squarefree in $y$}. 
It is easy to see (for example, by considering $S$-pair reductions) that every ideal that is squarefree in $y$ has a Gr\"obner basis, with respect to any $y$-compatible term order, such that $y^2$ does not divide any term of any element of the Gr\"obner basis. 

\begin{lemma}\label{lem:form}
If $I \subseteq R$ possesses a geometric vertex decomposition with respect to  a variable $y = x_j$ of $R$, then $I$ is squarefree in $y$, and the reduced Gr\"obner basis of $I$ with respect to any $y$-compatible term order has the form $\{yq_1+r_1,\dots, yq_k+r_k,h_1,\dots, h_\ell\}$ where $y$ does not divide any term of any $q_i$ or $r_i$ for any $ 1 \leq i \leq k$ nor any $h_j$ for any $1 \leq j \leq \ell$.
\end{lemma}

\begin{proof}
Fix a $y$-compatible term order $<$, and let $\mathcal{G} = \{y^{d_1}q_1+r_1,\dots, y^{d_m}q_m+r_m\}$ be the reduced Gr\"obner basis of $I$ with respect to $<$, where $y^{d_i}q_i = \text{in}_y(y^{d_i}q_i+r_i)$ and $y$ does not divide any term of $q_i$ for any $1 \leq i \leq m$.   Observe that $yq_i\in C_{y,I} \cap (N_{y,I}+\langle y \rangle)$ for each $1 \leq i \leq m$, so $\text{in}_y I = C_{y,I} \cap (N_{y,I}+\langle y \rangle)$ implies $yq_i\in \text{in}_y I$ for each $1 \leq i \leq m$. Hence, we may assume $d_i \leq 1$ for each $1 \leq i \leq m$.  
The remaining statements now follow easily.
\end{proof}

\subsection{Geometrically vertex decomposable ideals}
A geometric vertex decomposition of an ideal is analogous to a vertex decomposition of a simplicial complex into a deletion and star (see Remark \ref{rem:GVDsimplicial}). In this subsection, we extend this analogy by considering \emph{geometrically vertex decomposable ideals}, which are analogous to vertex decomposable simplicial complexes. We again let $R = \kappa[x_1,\dots, x_n]$ throughout this subsection. 
Recall that an ideal $I\subseteq R$ is \textbf{unmixed} if $\dim(R/P) = \dim(R/I)$ for all $P \in \text{Ass}(I)$.

\begin{definition}\label{def:geometricallyVertexDec}
An ideal $I\subseteq R$ is \textbf{geometrically vertex decomposable} if $I$ is unmixed and if
\begin{enumerate}
\item $I = \langle 1\rangle$ or $I$ is generated by indeterminates in $R$, or
\item for some variable $y = x_j$ of $R$, 
$\text{in}_y I =  C_{y,I}  \cap (N_{y,I}+\langle y\rangle)$ is a geometric vertex decomposition and the contractions of $N_{y,I}$ and $C_{y,I}$ to $\kappa[x_1,\dots,\widehat{y},\dots, x_n]$ are geometrically vertex decomposable. \qedhere
\end{enumerate}
\end{definition}

We take case $(1)$ to include the zero ideal, whose (empty) generating set vacuously consists only of indeterminates. 
We will soon need observations about the relative heights of the ideals $I$, $C_{y,I}$, and $N_{y,I} $ in the circumstances of condition $(2)$.  The degenerate cases are clear: if $C_{y,I} = \langle 1 \rangle$, then $\hgt(I) = \hgt(N_{y,I} )+1$ and, if $\sqrt{C_{y,I}} = \sqrt{N_{y,I}}$, then $\hgt(I) = \hgt(\text{in}_y I) = \hgt(C_{y,I})=\hgt(N_{y,I} )$. The nondegenerate case is handled by the lemma below.

 We say that the ring $R/I$ is \emph{equidimensional} if $\dim(R/P) = \dim(R/I)$ for all minimal primes $P$ of $I$ or, equivalently, if all irreducible components of the variety of $I$ have the same dimension.  Equidimensionality does not preclude the possibility that $I$ might have embedded primes and so is weaker than unmixedness.

\begin{lemma}\label{lem:height}
If $I\subseteq R$ is an ideal so that $R/I$ is equidimensional and $\text{in}_y I =  C_{y,I}  \cap (N_{y,I} + \langle y \rangle)$ is a nondegenerate geometric vertex decomposition with respect to some variable $y = x_j$ of $R$, then $\hgt(C_{y,I}) = \hgt(I) = \hgt(N_{y,I} )+1$.  Moreover, $R/C_{y,I}$ is equidimensional.
\end{lemma}

\begin{proof}
By Lemma \ref{lem:form}, $I$ has a reduced Gr\"obner basis of the form $\{yq_1+r_1, \ldots, yq_k+r_k, h_1, \ldots, h_\ell\}$ where $y$ does not divide any term of any $q_i$ or $r_i$, $1 \leq i \leq k$, nor any or $h_j$, $1 \leq j \leq \ell$. Let $\tilde{I}\subseteq R[t]$ be the ideal $\tilde{I} = \langle yq_1+tr_1, \ldots, yq_k+tr_k, h_1, \ldots, h_\ell\rangle$.  Using \cite[Theorem 15.17]{Eis95}, $R[t]/\tilde{I}\otimes_{\kappa[t]}\kappa[t,t^{-1}]\cong (R/I)[t,t^{-1}]$.  Clearly, $R/I$ equidimensional implies $(R/I)[t,t^{-1}]$ equidimensional, and so $R[t]/\tilde{I}\otimes_{\kappa[t]}\kappa[t,t^{-1}]$ is equidimensional. By \cite[Theorem 15.17]{Eis95}, $R[t]/\tilde{I}$ is flat as a $\kappa[t]$-module. By this flatness, $t$ is not a zero-divisor on $R[t]/\tilde{I}$, and so no minimal prime of $R[t]/\tilde{I}$ contains $t$.  Then because the primes of $R[t]/\tilde{I}\otimes_{\kappa[t]}\kappa[t,t^{-1}]$ are in correspondence with the primes of $R[t]/\tilde{I}$ that do not contain $t$, $R[t]/\tilde{I}$ is equidimensional as well.  Finally, because $R[t]/\langle \tilde{I}, t \rangle \cong R/\text{in}_y I$, it suffices to note that every minimal prime over $\langle t \rangle$ in $R[t]/ \tilde{I}$ has height exactly one as a consequence of Krull's principal ideal theorem and the fact that $t$ is not a zero-divisor on $R[t]/ \tilde{I}$.  Hence, $R/\text{in}_y I$ is equidimensional.

Because $\text{in}_y I =  C_{y,I}  \cap (N_{y,I} + \langle y \rangle)$, each minimal prime of $\text{in}_y I$ is either a minimal prime of $C_{y,I}$ or of  $N_{y,I} + \langle y \rangle$.  Conversely, each minimal prime of $N_{y,I}+\langle y \rangle$ either is a minimal prime of $\text{in}_y I$ or contains some minimal prime of $C_{y,I}$.  (Because $y \notin C_{y,I}$, no minimal prime of $C_{y,I}$ can contain any minimal prime of $N_{y,I}+\langle y \rangle$.)  Hence, because $\hgt(\text{in}_y I )= \hgt (I)$, we will have $ \hgt(C_{y,I}) = \hgt(I) = \hgt(N_{y,I} )+1$ so long as some minimal prime of $N_{y,I}+\langle y \rangle$ does not contain a minimal prime of $C_{y,I}$, i.e., so long as $\sqrt{C_{y,I}} \not\subseteq \sqrt{N_{y,I}+\langle y \rangle} $.  Because $C_{y,I}$ and $N_{y,I}$ have generating sets that do not involve $y$, we cannot have $\sqrt{C_{y,I}} \subseteq \sqrt{N_{y,I}+\langle y \rangle}$ unless $\sqrt{C_{y,I}} \subseteq \sqrt{N_{y,I}}$.  But $N_{y,I} \subseteq C_{y,I}$, so $\sqrt{C_{y,I}} \subseteq \sqrt{N_{y,I}}$ would imply $\sqrt{C_{y,I}} = \sqrt{N_{y,I}}$, contradicting the assumption of nondegeneracy.

Finally, because every minimal prime of $\sqrt{C_{y,I}}$ is a minimal prime of $\text{in}_y I$, equidimensionality of $R/C_{y,I}$ follows from equidimensionality of $R/\text{in}_y I$.  
\end{proof}

As noted above, the definition of a geometrically vertex decomposable ideal is analogous to the definition of a vertex decomposable simplicial complex. In particular, we have the following proposition, whose proof we leave as an exercise:

\begin{proposition}\label{prop:VDiffGVD}
Let $\Delta$ be a simplicial complex on vertex set $[n]$. Its Stanley--Reisner ideal $I_\Delta \subseteq R$ is geometrically vertex decomposable if and only if $\Delta$ is vertex decomposable.
\end{proposition}

In the remainder of this section, we discuss some properties of geometrically vertex decomposable ideals and further connections to vertex decomposable simplicial complexes.

\begin{proposition}\label{prop:radical}
A geometrically vertex decomposable ideal is radical.
\end{proposition}

\begin{proof}
Let $I \subseteq R$ be a geometrically vertex decomposable ideal.  We proceed by induction on $n = \dim(R)$.  We note first that if $I = \langle 0\rangle$, $I = \langle 1\rangle$, or $I$ is generated by indeterminates, the result is immediate.  
Otherwise, there exists some variable $y = x_j$ such that $\text{in}_y I = C_{y,I}  \cap (N_{y,I}+\langle y\rangle)$ is a geometric vertex decomposition and the contractions of  $C_{y,I}$ and $N_{y,I}$ to $\kappa[x_1,\dots, \widehat{y},\dots, x_n]$ are geometrically vertex decomposable. These contracted ideals are radical by induction, thus so are $C_{y,I}$ and $N_{y,I}$. Hence, $\text{in}_y I = C_{y,I}  \cap (N_{y,I}+\langle y\rangle)$ is radical. Finally, because $\text{in}_y I$ is radical, $I$ must also be radical.
\end{proof}

We remark briefly that Proposition \ref{prop:radical} does not require the unmixedness assumptions on $I$, $C_{y,I}$, or $N_{y,I}$. 

We next consider geometrically vertex decomposable ideals that have a certain compatibility with a given lexicographic monomial order. The main result in our discussion of these ideals is Proposition \ref{prop:simplicial}, which we will need in Section \ref{sect:applications} on applications.  
\begin{definition} 
Fix a lexicographic monomial order $<$ on $R$. 
We say that an ideal $I\subseteq R$ is \textbf{$<$-compatibly geometrically vertex decomposable} if $I$ satisfies Definition \ref{def:geometricallyVertexDec} upon replacing item (2) with
\begin{enumerate}
\item[(2*)] for the $<$-largest variable $y$ in $R$, $\text{in}_{y} I = C_{y,I} \cap (N_{y,I}+\langle y\rangle)$ is a  
geometric vertex decomposition and the contractions of $N_{y,I}$ and $C_{y,I}$ to $\kappa[x_1,\dots,\widehat{y},\dots, x_n]$ are $<$-compatibly geometrically vertex decomposable  for the naturally induced monomial order on $\kappa[x_1,\dots, \widehat{y},\dots, x_n]$ (which we also call $<$).\qedhere
\end{enumerate}
\end{definition}

Let $\Delta$ be a simplicial complex on a vertex set $[n]$,  
and let $<$ be a total order on $[n]$. 
We say that a simplicial complex $\Delta$ is \textbf{$<$-compatibly vertex decomposable} if either $\Delta = \emptyset$ or $\Delta$ is a simplex or, for the $<$-largest vertex $v\in \Delta$, $\text{del}_{\Delta}(v)$ and $\text{lk}_{\Delta}(v)$ are $<$-compatibly vertex decomposable.

The following is an easy consequence of \cite[Theorem 2.1]{KMY09}. 
\begin{lemma}\label{lem:compatible}
Suppose that $I\subseteq R$ 
is squarefree in $y = x_j$, and suppose that $<$ is a $y$-compatible monomial order on $R$. Then $\text{in}_< I = \text{in}_< C_{y,I}\cap (\text{in}_< N_{y, I}+\langle y\rangle)$. 
\end{lemma}

\begin{proof}

Since $I$ is squarefree in $y$, $I$ has a Gr\"obner basis $\{yq_1+r_1,\dots, yq_k+r_k,h_1,\dots, h_\ell\}$ where $y$ does not divide any term of $q_i$, $r_i$, $1 \leq i \leq k$, nor $h_j$, $1 \leq j \leq \ell$. 
 Let $m_i = \text{in}_<q_i$, $1\leq i\leq k$, and $m_{k+i} = \text{in}_< h_{i}$, $1\leq i\leq \ell$. 
By \cite[Theorem 2.1(a)]{KMY09}, $\{q_1,\dots, q_{k}, h_1,\dots, h_{\ell}\}$ and $\{h_1,\dots, h_\ell\}$ are Gr\"obner bases for $C_{y,I}$ and $N_{y,I}$ respectively, so $\text{in}_< C_{y,I} = \langle m_i\mid 1\leq i\leq {k+\ell}\rangle$ and $\text{in}_< N_{y,I} = \langle m_{k+i}\mid 1\leq i\leq \ell\rangle$. It is then straightforward to check 

\begin{equation}\label{eq:initialDecomp}
\text{in}_< I = \langle ym_1,\dots, ym_k, m_{k+1},\dots, m_{k+\ell}\rangle = \text{in}_< C_{y,I}\cap (\text{in}_< N_{y,I}+\langle y\rangle). \qedhere
\end{equation}
\end{proof}

We are now ready to prove the main result of our discussion on $<$-compatibly geometrically vertex decomposable ideals.  

\begin{proposition}\label{prop:simplicial}

An ideal $I\subseteq R$ is $<$-compatibly geometrically vertex decomposable for the lexicographic monomial order $x_1>x_2>\cdots>x_n$ if and only if $\text{in}_< I$ is the Stanley--Reisner ideal of a $<$-compatibly vertex decomposable simplicial complex on $[n]$ for the vertex order $1>2>\cdots>n$. 
\end{proposition}

\begin{proof}

If $I = \langle 1\rangle$ or $\langle 0\rangle$, there is nothing to show. So suppose that $I$ is nontrivial and proceed by induction on $n = \text{dim}(R)$. The base case $n = 1$ is straightforward. 

Suppose that $n\geq 2$ is arbitrary. First assume that $I$ is $<$-compatibly geometrically vertex decomposable and let $y = x_1$. 
If $I$ is generated by indeterminates, there is nothing to show. Otherwise, we have that $\text{in}_{y} I = C_{y, I}\cap (N_{y, I}+\langle y\rangle)$ is a geometric vertex decomposition, and the contractions $N^c$ and $C^c$ of $N_{y, I}$ and $C_{y, I}$ to $\kappa[x_2,\dots, x_n]$ are $<$-compatibly geometrically vertex decomposable. 
There are two cases.

The first case is when $C_{y,I} = \langle 1 \rangle$, which implies that $\text{in}_< I = \text{in}_<N_{y,I}+\langle y\rangle$. By induction, $\text{in}_<N^c$ is the Stanley--Reisner ideal of a $<$-compatibly vertex decomposable simplicial complex. 
Thus, $\text{in}_<I$, which is equal to $\text{in}_<N_{y,I}+\langle y\rangle$, is too. Indeed, the complexes $\Delta(\text{in}_<N^c)$ and $\Delta(\text{in}_<N_{y,I}+\langle y\rangle)$ are the same (though on different ambient vertex sets).

 Now assume $C_{y, I} \neq \langle 1 \rangle$.  By induction, $\text{in}_< N^c$ and $\text{in}_< C^c$ are the Stanley--Reisner ideals of $<$-compatibly vertex decomposable simplicial complexes, thus so are $\text{in}_< N_{y, I}+\langle y\rangle$ and $\text{in}_< C_{y, I}+ \langle y\rangle$. By Lemma \ref{lem:compatible}, we have
\begin{equation}\label{eq:squarefreeDecompCN}
\text{in}_< I = \text{in}_<C_{y,I}\cap(\text{in}_<N_{y,I}+\langle y\rangle). 
\end{equation}
Thus, $\text{in}_< I$ is a squarefree monomial ideal. 
Let $\Delta := \Delta(\text{in}_<I)$. 
Equation \eqref{eq:initialDecomp} and Lemma \ref{lem:starDelLink} imply that $\text{in}_<C_{y,I}$ and $\text{in}_<N_{y,I}+\langle y\rangle$ are the Stanley--Reisner ideals of $\text{star}_\Delta(1)$ and $\text{del}_\Delta(1)$.
Thus $\text{in}_<C_{y,I}+\langle y\rangle$ and $\text{in}_<N_{y,I}+\langle y\rangle$ are the Stanley--Reisner ideals of $\text{lk}_\Delta(1)$ and $\text{del}_\Delta(1)$. 
 Hence $\text{lk}_\Delta(1)$ and $\text{del}_\Delta(1)$ are $<$-compatibly vertex decomposable. Thus, $\Delta$ is too.

For the converse, assume that $\Delta = \Delta(\text{in}_< I)$ is $<$-compatibly vertex decomposable. 

Since $y = x_1$ is $<$-largest, and $\text{in}_< I$ is a squarefree monomial ideal by assumption, the reduced Gr\"obner basis of $I$ with respect to $<$ is squarefree in $y$, and so has the form $\mathcal{G} = \{yq_1+r_1,\dots, yq_k+r_k,h_1,\dots, h_\ell\}$, with  $y$ not dividing any term of any $q_i, r_i, h_j$, $1\leq i\leq k$, $1\leq j\leq \ell$. So, by \cite[Theorem 2.1 (b)]{KMY09}, $\text{in}_{y}I = C_{y,I}\cap (N_{y,I}+\langle y\rangle)$, is a geometric vertex decomposition. Again, we have two cases.

If $C_{y,I} = \langle 1 \rangle$, then $\text{in}_<I = \text{in}_< N_{y,I}+\langle y\rangle$. The complex $\Delta(N_{y,I}+\langle y\rangle)$ is the same as $\Delta(N^c)$. Thus, $\Delta(N^c)$ is $<$-compatibly vertex decomposable. So, by induction, $N^c$ is $<$-compatibly geometrically vertex decomposable. Hence, so too is $I$.

 If $C_{y,I} \neq \langle 1 \rangle$, then we have that $1$ is a vertex of $\Delta$ and, as discussed above, $\text{in}_<C_{y,I}+\langle y\rangle$ and $\text{in}_<N_{y,I}+\langle y\rangle$ are the Stanley--Reisner ideals of $\text{lk}_\Delta(1)$ and $\text{del}_\Delta(1)$. 
Since $\Delta$ is $<$-compatibly vertex decomposable, so are $\text{lk}_\Delta(1)$ and $\text{del}_\Delta(1)$. Thinking of these complexes as complexes on vertex set $\{2,3,\dots, n\}$, their Stanley--Reisner ideals are $\text{in}_<C^c$ and $\text{in}_<N^c$. So, by induction, $C^c$ and $N^c$ are $<$-compatibly geometrically vertex decomposable. Hence, so too is $I$.
\end{proof}

We end this section with an example that shows that there exist geometrically vertex decomposable ideals that are \emph{not} geometrically vertex decomposable compatible with any lexicographic monomial order. 

\begin{example}\label{ex:nolex}
Let $I = \langle y(zs-x^2), ywr, wr(z^2+zx+wr+s^2) \rangle \subseteq \kappa[x,y,z,w,r,s]$. 
Observe that $I$ is squarefree in $y$, and we have a geometric vertex decomposition with $C_{y,I} = \langle zs-x^2,wr \rangle$ and $N_{y,I} = \langle (wr)(zx+s^2+z^2+wr) \rangle$. 
Furthermore, the contractions of $C_{y,I}$ and $N_{y,I}$ to $\kappa[x,z,w,r,s]$ are geometrically vertex decomposable. (To see this, let $C^c$ and $N^c$ denote these contracted ideals. Then $C^c$ and $N^c$ are squarefree in $s$ and $x$, respectively, and $\text{in}_s C^c = \langle zs, wr\rangle$ and $\text{in}_x N^c = \langle wrzx\rangle$.) Hence $I$ is geometrically vertex decomposable. 

Next, we will observe that $I$ has no squarefree initial ideals, hence cannot be $<$-compatibly geometrically vertex decomposable for any order $<$ by Proposition \ref{prop:simplicial}. To prove this, we first note that the given generating set $\{g_1 := y(zs-x^2), g_2 := ywr, g_3 := wr(z^2+zx+wr+s^2)\}$ of $I$ is a universal Gr\"obner basis. Indeed, fix an arbitrary monomial order, and observe that each $S$-polynomial $S(g_i,g_j)$, $i\neq j$, is divisible by $g_2 = ywr$ and thus reduces to $0$ under division by $\{g_1,g_2,g_3\}$. Consequently, if $I$ has a squarefree initial ideal, there exists a monomial order $<$ such that $\langle \text{in}_< g_1, \text{in}_< g_2, \text{in}_< g_3\rangle$ is a squarefree monomial ideal. Noting that none of the monomials of $g_i$ are divisible by any of the monomials in $g_j$, it follows that $\text{in}_< g_1, \text{in}_<g_2, \text{in}_< g_3$ are minimal generators for $\langle \text{in}_< g_1, \text{in}_<g_2, \text{in}_< g_3\rangle$. Hence, the only way for $\langle \text{in}_< g_1, \text{in}_<g_2, \text{in}_< g_3\rangle$ to be a squarefree monomial ideal is if each of $ \text{in}_< g_1, \text{in}_<g_2, \text{in}_< g_3$ is a squarefree monomial, which would force (i) $\text{in}_< (zs-x^2) = zs$ and (ii) $\text{in}_< (z^2+zx+wr+s^2) = zx$. So, suppose there is some monomial order $<$ that satisfies (i) and (ii). Then we have $zx > z^2$ (by (ii)) and hence $x>z$. We also have $zs>x^2$ (by (i)). So, since $x>z$, we have $zs>x^2>zx$ and so $s>x$. Finally, $zx>s^2$ (by (ii)) together with $s>x$ implies that $zx>s^2>x^2$. Hence $z>x$, which is impossible as we have already concluded that $x>z$. Thus no monomial order $<$ that satisfies both (i) and (ii) exists, and there is no squarefree initial ideal of $I$.
\end{example}

\section{Background on Liaison}\label{sect:liaison}
In this section, we recall some background on liaison theory. The first subsection concerns terminology and results relevant to our work. The second subsection provides further context for our second goal from the introduction. 

\subsection{Liaison theory basics}
Here we review standard definitions and lemmas on Gorenstein liaison theory that we will need in this paper. For a more thorough introduction, see \cite{MN01}.  
We follow definitions and some notation from \cite{GMN13}, which provides a careful discussion of how liaison theory can be used to make inferences about Gr\"obner bases. Throughout this subsection, we let $R = \kappa[x_0, x_1,\dots, x_n]$ with the standard grading.

\begin{definition}
Let $V_1, V_2, X \subseteq \mathbb{P}^{n}$ be subschemes defined by saturated ideals $I_{V_1}$, $I_{V_2}$, and $I_{X}$ of $R$, respectively, and assume that $X$ is arithmetically Gorenstein.  If $I_X \subseteq I_{V_1} \cap I_{V_2}$ and if $[I_X:I_{V_1}] = I_{V_2}$ and $[I_X: I_{V_2}] = I_{V_1}$, then $V_1$ and $V_2$ are \textbf{directly algebraically $G$-linked} by $X$, and we write $I_{V_1} \sim I_{V_2}$.
\end{definition}

One may generate an equivalence relation using these direct links.

\begin{definition} If there is a sequence of links $V_1 \sim \cdots \sim V_k$ for some $k \geq 2$, then we say that $V_1$ and $V_k$ are in the same \textbf{$G$-liaison class (or Gorenstein liaison class)} and that they are \textbf{$G$-linked} in $k-1$ steps.  Of particular interest is the case in which $V_k$ is a complete intersection, in which case we say that $V_1$ is \textbf{in the Gorenstein liaison class of a complete intersection (abbreviated glicci)}.
\end{definition}

We will say that a homogeneous, saturated, unmixed ideal of $R$ is glicci if it defines a glicci subscheme of $\mathbb{P}^{n}$.  It is because liaison was developed to study subschemes of projective space that the restriction to homogeneous, saturated ideals is natural.  Throughout this paper, we will be interested in $G$-links coming from \emph{elementary $G$-biliaisons}. Indeed, it is through elementary $G$-biliaisons that we connect geometric vertex decomposition to liaison theory. 

Let $S$ be a ring.  If $S_P$ is Gorenstein for all prime ideals $P$ of height $0$, then we say that $S$ is $\mathbf{G_0}$.  

\begin{definition}\label{def:Gbiliaison}
Let $I$ and $C$ be homogeneous, saturated, unmixed ideals of $R $ with $\hgt(I) = \hgt(C)$.  Suppose there exist $\ell \in \mathbb{Z} $, a homogeneous Cohen--Macaulay ideal $N \subseteq I \cap C$ of height $\hgt(I)-1$, and an isomorphism $I/N \cong [C/N](-\ell)$ as graded $R/N$-modules.  If $N$ is $G_0$, then we say that $I$ is obtained from $C$ by an \textbf{elementary $G$-biliaison of height $\ell$}.  
\end{definition}

\begin{theorem}\cite[Theorem 3.5] {Har07}
Let $I$ and $C$ be homogeneous, saturated, unmixed ideals defining subschemes $V_I$ and $V_C$, respectively, of $\mathbb{P}^{n}$.  If $I$ is obtained from $C$ by an elementary $G$-biliaison, then $V_I$ is $G$-linked to $V_C$ in two steps.
\end{theorem}

\begin{remark}
\textbf{Even $G$-liaison classes} are equivalence classes of subschemes of $\mathbb{P}^{n}$ of a fixed codimension that are $G$-linked to one another in an even number of steps.  
Two subschemes in the same \emph{even} $G$-liaison class are more closely related to one another than are two subschemes that can be linked to one another but only in an odd number of steps (see \cite[Section 3]{Nag98}). 
This provides some intuition for why various classes of generalized determinantal varieties can be linked to one another in an even number of steps, via elementary $G$-biliaisons, yet there is no reason to expect that the intermediate links share a similar form, or are at all easy to describe.
\end{remark}

\subsection{Further context on a question in liaison theory}

The purpose of this subsection is to recall the motivation for the following question:  
\begin{quest}\cite[Question 1.6]{KM+01}\label{motivQuest}
Is every arithmetically Cohen--Macaulay subscheme of $\mathbb{P}^n$ glicci?
\end{quest} 

Because all complete intersections of a fixed codimension are in the same liaison class, an equivalent formulation of the question is 

\begin{quest}
For each codimension, is there exactly one Gorenstein liaison class containing any (or all) Cohen--Macaulay subschemes of $\mathbb{P}^n$?
\end{quest}

 This question arises by analogy to the special case of complete intersection liaison in codimension 2, where all arithmetically Cohen--Macaulay subschemes are in the (complete intersection and so also Gorenstein) liaison class of a complete intersection \cite[Theorem 3.2]{PS74}.  The same is not true in higher codimensions, however.  In fact, in higher codimensions there are infinitely many complete intersection liaison classes containing arithmetically Cohen--Macaulay schemes (see \cite[Chapter 7]{KM+01} and, for related ideas, \cite{HU87}).  Complete intersection liaison is a well-understood and very satisfying theory in codimension 2, and this failure to generalize to higher codimensions suggests that it is worth searching for a theory that reduces to complete intersection liaison in codimension 2 and also preserves many of its desirable properties in higher codimension.  This is one of the motivations for studying Gorenstein liaison, where better control of the liaison classes containing arithmetically Cohen--Macaulay schemes may still be hoped for in all codimensions.  In particular, an affirmative answer to Question \ref{motivQuest} would serve, at least in the eyes of some, as an endorsement of the structure of Gorenstein liaison.

There are partial results in the direction of an affirmative answer to Question \ref{motivQuest}, including the results that standard determinantal schemes \cite[Theorem 1.1]{KM+01}, mixed ladder determinantal schemes from two-sided ladders \cite[Corollary 2.2]{Gor07}, schemes of Pfaffians \cite[Theorem 2.3]{DG09}, wide classes of arithmetically Cohen--Macaulay curves in $\mathbb{P}^4$ \cite{CMR00, CMR01}, and arithmetically Cohen-Macaulay schemes defined by Borel-fixed monomial ideals \cite[Theorem 3.5]{MN02} are all glicci. For more results, see \cite{Cas03, HSS08}.

There have also been some quite general discoveries.  M. Casanellas, E. Drozd, and R. Hartshorne \cite{CDH05} gave a general characterization of when two subschemes of a normal arithmetically Gorenstein scheme are in the same Gorenstein liaison class and showed that every arithmetically Gorenstein subscheme of $\mathbb{P}^n$ is glicci.   In \cite[Theorem 3.1]{Gor08}, E. Gorla obtained the very broad result that every determinantal scheme is glicci, generalizing the results of \cite[Theorem 1.1]{KM+01} and also \cite[Theorem 4.1]{Har07}.  Later, J. Migliore and U. Nagel \cite{MN13} showed that every arithmetically Cohen-Macaulay subscheme of $\mathbb{P}^n$ that is generically Gorenstein is actually glicci when viewed as a subscheme of $\mathbb{P}^{n+1}$.

One can find both encouragement and cause for trepidation in \cite{Har02}: R. Hartshorne gave positive results for many sets of points in $\mathbb{P}^3$ and curves in $\mathbb{P}^4$ but also produced still-viable candidates for a source of a negative answer.  The precision required to study Hartshorne's examples highlights the complexity of Question \ref{motivQuest}. 

By connecting geometric vertex decomposition and liaison, we provide more evidence in favor of an affirmative answer to this question and give a framework for assessing membership in the Gorenstein liaison class of a complete intersection for some arithmetically Cohen--Macaulay schemes arising naturally from combinatorial data.

\section{(Weakly) Geometrically vertex decomposable ideals are glicci}\label{sect:GVDGlicci}

In Section \ref{sect:GVDtoGBiliaison}, we show that under mild hypotheses a geometric vertex decomposition gives rise to an elementary $G$-biliaison of height $1$ (Corollary \ref{cor:gvdToLia}). We use this result in Subsection \ref{sect:GVDtoGlicci} to prove that every geometrically vertex decomposable ideal is glicci (Theorem \ref{thm:glicci}). We also define the class of \emph{weakly geometrically vertex decomposable ideals}, a class that contains the geometrically vertex decomposable ideals, and we prove that each weakly geometrically vertex decomposable ideal is glicci (Corollary \ref{weakGVDimpliesgliggi}). Finally, in Subsection \ref{sect:GBapplications}, we obtain some consequences on Gr\"obner bases and Gr\"obner degenerations. Throughout this section, we assume that the field $\kappa$ is infinite, and we let $R$ denote the standard graded polynomial ring $\kappa[x_1, \ldots, x_n]$. 

\subsection{An elementary $G$-biliaison arising from a geometric vertex decomposition}\label{sect:GVDtoGBiliaison}

We begin by using a geometric vertex decomposition to construct the isomorphism that will constitute an elementary $G$-biliaison when the setting is appropriate.

\begin{theorem}\label{thm:onestep}
Suppose that $I \subseteq R$ is an unmixed ideal possessing a nondegenerate geometric vertex decomposition with respect to some variable $y = x_j$ of $R$.  If $N_{y,I} $ is unmixed, then there is an isomorphism $I/N_{y,I} \cong C_{y,I}/N_{y,I} $ as $R/N_{y,I}$-modules.  If $N_{y,I}$, $C_{y,I}$, and $I$ are homogeneous, then the same map is an isomorphism $I/N_{y,I} \cong [C_{y,I}/N_{y,I} ](-1)$ in the category of graded $R/N_{y,I}$-modules.
\end{theorem}
\begin{proof}
Fix a $y$-compatible term order $<$.  From Lemma \ref{lem:form}, we know that the reduced Gr\"obner basis of $I$ has the form $\mathcal{G} = \{yq_1+r_1,\dots, yq_k+r_k,h_1,\dots, h_\ell\}$ where $y$ does not divide any term of $q_i$ or of $r_i$ for any $1 \leq i \leq k$ nor any $h_j$ for $1 \leq j \leq \ell$.   Let $C = C_{y,I} = \langle q_1,\dots, q_k, h_1,\dots, h_\ell\rangle$ and $N = N_{y,I} = \langle h_1,\dots, h_\ell \rangle$. 

We first observe that $N\subseteq I \cap C$.  To build the desired isomorphism, we will need to find regular elements of $R/N$.  Towards that end, we claim that $\langle q_1, \ldots, q_k \rangle \not\subseteq Q$ for any minimal prime $Q$ of $N$.  If it were, then we would also have $C \subseteq Q$, which is impossible because $\hgt(Q) = \hgt(N)<\hgt(C)$ by Lemma \ref{lem:height}.  Similarly, $\langle yq_1+r_1, \ldots, yq_k+r_k \rangle \not\subseteq Q'$ for any minimal prime $Q'$ of $N$ since, if it were, then we would have $I \subseteq Q'$, in violation of Lemma \ref{lem:height}.  Because $\kappa$ is infinite, we may choose scalars $a_1, \ldots, a_k \in \kappa$ so that neither $u:=a_1q_1+\cdots+a_kq_k$ nor $v:=a_1(yq_1+r_1)+ \ldots+ a_k(yq_k+r_k)$ is an element of any minimal prime of $N$.  Because $\min(N) = \ass(N)$, neither $u$ nor $v$ is a zero-divisor on $R/N$.  

We may now define a map $\phi: C \rightarrow I/N$ given by $f \mapsto \dfrac{fv}{u}$.  To see that $\phi$ is well defined, we claim that, for each $f \in C$, there exists a unique $\overline{g} \in I/N$ so that $fv-gu \in N$ (where $\overline{g}$ is the class of $g$ in $I/N$).  Suppose that $fv-g_1u = n_1 \in N$ and $fv-g_2u = n_2 \in N$ for some $g_1, g_2 \in I$.  Then $(g_1-g_2)u = n_2-n_1 \in N$, and so, because $u$ is not a zero-divisor on $I/N$, $g_1-g_2 \in N$.  Hence, there is at most one such $\overline{g} \in I/N$.  To see that there is at least one such $\overline{g} \in I/N$, we note that $\overline{g} = \overline{0}$ is a satisfying choice if $f \in N$ and claim that $\overline{g} = \overline{yq_i+r_i}$ is a satisfying choice if $f = q_i$ for each $1 \leq i \leq k$.  Indeed, \begin{align*}
(yq_i+r_i)u-q_iv = &(yq_i+r_i)(a_1q_1+\cdots+a_kq_k) - q_i(a_1(yq_1+r_1)+\cdots+a_k(yq_k+r_k)) \\
= &r_i(a_1q_1+\cdots+a_kq_k)-q_i(a_1r_1+\cdots+a_kr_k).
\end{align*}
Because $yq_i+r_i \in I$ and $v \in I$, we have $r_i(a_1q_1+\cdots+a_kq_k)-q_i(a_1r_1+\cdots+a_kr_k) \in I$.  But $y$ does not divide any term of any $q_j$ or any $r_j$ for $1 \leq j \leq k$, and so the leading term of $r_i(a_1q_1+\cdots+a_kq_k)-q_i(a_1r_1+\cdots+a_kr_k)$ is not divisible by the leading term of any $yq_i+r_i$.  By the assumptions that $\mathcal{G}$ is a Gr\"obner basis of $I$ and that $<$ is $y$-compatible, it must be that $r_i(a_1q_1+\cdots+a_kq_k)-q_i(a_1r_1+\cdots+a_kr_k)$ has a Gr\"obner basis reduction using only the elements of $\mathcal{G}$ not involving $y$, i.e., using only $h_1,\dots, h_\ell $, which implies that $r_i(a_1q_1+\cdots+a_kq_k)-q_i(a_1r_1+\cdots+a_kr_k) \in N$.  Hence, multiplication by $v$ gives a map from $C$ to $u(I/N)$, which maps isomorphically to $I/N$ by multiplication by $1/u$.  That is, $\phi$ is indeed a map from $C$ to $I/N$.  Because $I$ is generated over $N$ by the $yq_i+r_i$, we have also shown that $\phi$ is surjective.

Having established that $\phi(h_i) = \overline{0} \in I/N$, we have $N \subseteq \ker(\phi)$.  And $\ker(\phi) \subseteq N$ because $v$ is a non-zero-divisor on $R/N$.  Therefore, $\phi$ induces an isomorphism $\overline{\phi}: C/N \rightarrow I/N$.  It is clear that whenever $N$ is homogeneous so that discussion of degrees makes sense, $\overline{\phi}$ increases degree by $1$ and so  $\overline{\phi}: [C/N](-1) \rightarrow I/N$ is an isomorphism of graded $R/N$-modules.  \end{proof}

Notice that if, in the proof above, one already knows $q_1$ and $yq_1+r_1$ to be non-zero-divisors on $R/N$, for example if $R/N$ is a domain, one may choose $a_1 = 1$ and $a_i = 0$ for $1<i \leq k$.  In this case, the map $\phi$ will be of the same form used in \cite{Gor07}, \cite{Gor08}, and \cite{GMN13}.

As indicated above, the primary use of Theorem \ref{thm:onestep} is in the setting of liaison theory (Corollary \ref{cor:gvdToLia}). We will need the following straightforward fact about saturation:

\begin{lemma}\label{lem:saturated}
If $I \subseteq R$ is homogeneous and unmixed, then $\sqrt{I}$ is the homogeneous maximal ideal $m$ or $I$ is saturated.  
\end{lemma}
\begin{proof}

Observe that $m$ is an associated prime of $I$ if and only if it is a minimal prime of $I$ if and only if $\sqrt{I} = m$.  
\end{proof}

\begin{corollary}\label{cor:gvdToLia}
Let $I$ be a homogeneous, saturated, unmixed ideal of $R$ and $\text{in}_y I = C_{y,I} \cap (N_{y,I}+\langle y \rangle)$ a nondegenerate geometric vertex decomposition with respect to some variable $y = x_j$ of $R$.  Assume that $N_{y,I}$ is Cohen--Macaulay and $G_0$ and that $C_{y,I}$ is also unmixed.  Then $I$ is obtained from $C_{y,I}$ by an elementary $G$-biliaison of height $1$.
\end{corollary}

\begin{proof}
The height conditions required by the definition of elementary $G$-biliaison are given by Lemma \ref{lem:height}, saturation follows from Lemma \ref{lem:saturated}, and the required isomorphism is constructed in Theorem \ref{thm:onestep}.
\end{proof}

\subsection{Geometrically vertex decomposable ideals and the glicci property}\label{sect:GVDtoGlicci}

We make two observations about linkage before proceeding. Let $S = R[z]$ for a new variable $z$. \begin{enumerate}
\item If $I$ is obtained from $C$ via an elementary $G$-biliaison of height $1$ in $R$, then $IS$ is obtained from $CS$ via an elementary $G$-biliaison of height $1$ in $S$. \label{Observation1}
\item If $I$ is obtained from $C$ via an elementary $G$-biliaison of height $1$ in $R$, then $IS+\langle z \rangle$ is obtained from $CS+\langle z \rangle$ via an elementary $G$-biliaison of height $1$ in $S$. \label{Observation2}
\end{enumerate}

\begin{theorem}\label{thm:glicci}
If $I = I_0 \subseteq R$ is a homogeneous, geometrically vertex decomposable proper ideal, then there is a finite sequence of homogeneous, saturated, unmixed ideals $I_1, \ldots, I_t$ so that $I_{j-1}$ is obtained from $I_{j}$ by an elementary $G$-biliaison of height $1$ for every $1 \leq j \leq t$ and $I_t$ is a complete intersection.  In particular, $I$ is glicci.
\end{theorem}

\begin{proof} 
Clearly, it suffices to prove the first claim.  We will proceed by induction on $n = \dim(R)$, noting that the case of a dimension $0$ polynomial ring is trivial.

We now take $n \geq 1$ to be arbitrary and assume the result for all proper homogeneous ideals $I$ in polynomial rings of dimension $<n$.  If $I$ is a complete intersection, then there is nothing to prove. Otherwise, there exists some variable $y = x_j$ of $R$ for which $\text{in}_y I =  C_{y,I}  \cap (N_{y,I} +\langle y\rangle)$ is a geometric vertex decomposition with the contractions of $N_{y,I} $ and $C_{y,I}$ to $T = \kappa[x_1,\ldots, \hat{y}, \ldots, x_n]$ geometrically vertex decomposable.  

Suppose first that $ C_{y,I}  = \langle 1 \rangle$, in which case $I = N_{y,I} +\langle y\rangle$ (possibly after a linear change of variables).  By induction, with $\tilde{I}_0 = N_{y,I} \cap T$, there is a sequence of ideals $\tilde{I}_1, \ldots, \tilde{I}_t$ of $T$ so that $\tilde{I}_{j-1}$ is obtained from $\tilde{I}_{j}$ by an elementary $G$-biliaison of height $1$ for every $1 \leq j \leq t$ and $\tilde{I}_t$ is a complete intersection.  Setting $I_j = \widetilde{I}_jR+\langle y \rangle$ for every $1 \leq j \leq t$, the result follows from Observation \eqref{Observation2}, above.  Similarly, the result is essentially immediate from the inductive hypothesis together with Observation \eqref{Observation1} in the other degenerate case $\sqrt{C_{y,I}} = \sqrt{N_{y,I}}$:  indeed, because $I$ is radical by Proposition \ref{prop:radical}, we have $I = C_{y,I} = N_{y,I}$ by Proposition \ref{prop:degen-rad}.

Finally, assume the geometric vertex decomposition with respect to $I$ is nondegenerate, in which case we may apply the inductive hypothesis to $\tilde{I}_1 = C_{y,I} \cap T$.  By induction and in parallel with the previous case, there is a finite sequence of homogeneous, saturated, unmixed ideals $\widetilde{I}_2, \ldots, \widetilde{I}_t$ of $T$ so that $\widetilde{I}_{j-1}$ is obtained from $\widetilde{I}_j$ by an elementary $G$-biliaison of height $1$ in $T$ for every $2 \leq j \leq t$ and $\widetilde{I}_t$ is a complete intersection.  Let $I_j = \widetilde{I}_jR$ for every $2 \leq j \leq t$.  Then with $I_1 = C_{y,I}$, by Observation \eqref{Observation1} above, $I_{j-1}$ is obtained from $I_{j}$ by an elementary $G$-biliaison of height $1$ in $R$ for every $2 \leq j \leq t$ and $I_t$ is a complete intersection.  Hence, it suffices to show that $I$ is obtained from $C_{y,I}$ by an elementary $G$-biliaison of height $1$, but this is Corollary \ref{cor:gvdToLia}. Note that the hypotheses of Corollary \ref{cor:gvdToLia} hold since $C_{y,I}\cap T$ and $N_{y,I}\cap T$ are geometrically vertex decomposable (hence unmixed and radical, and so $G_0$) and glicci (hence Cohen-Macaulay) by induction.
\end{proof}

\begin{corollary}\label{cor:CM}
If $I \subseteq R$ is a homogeneous, geometrically vertex decomposable proper ideal, then $I$ is Cohen--Macaulay.
\end{corollary}
\begin{proof}
Since glicci implies Cohen--Macaulay, this is immediate from Theorem \ref{thm:glicci}.
\end{proof}

The remainder of this section will concern weakly geometrically vertex decomposable ideals, a direct generalization of the monomial ideals associated to weakly vertex decomposable simplicial complexes in the sense of \cite{NR08} (see Remark \ref{rem:weakVert} below).

\begin{definition}\label{def:weaklyGeometricallyVertexDec}
An ideal $I\subseteq R$ is \textbf{weakly geometrically vertex decomposable} if $I$ is unmixed and if
\begin{enumerate}
\item $I = \langle 1\rangle$ or $I$ is generated by indeterminates in $R$, or
\item (\emph{degenerate case}) for some variable $y = x_j$ of $R$, 
$\text{in}_y I = C_{y,I} \cap (N_{y,I}+\langle y \rangle)$ is a degenerate geometric vertex decomposition and the contraction of $N_{y,I}$ to the ring $\kappa[x_1,\dots,\widehat{y},\dots, x_n]$ is weakly geometrically vertex decomposable, or

\item (\emph{nondegenerate case}) for some variable $y = x_j$ of $R$, $\text{in}_y I = C_{y,I} \cap (N_{y,I}+\langle y \rangle)$ is a nondegenerate geometric vertex decomposition, the contraction of $C_{y,I}$ to the ring $\kappa[x_1,\dots,\widehat{y},\dots, x_n]$ is weakly geometrically vertex decomposable, and $N_{y,I}$ is radical and Cohen--Macaulay. \qedhere
\end{enumerate}
\end{definition}

Notice that it makes no difference whether we require $N_{y,I}$ or the contraction of $N_{y,I}$ to $\kappa[x_1,\dots,\widehat{y},\dots, x_n]$ to be radical and Cohen--Macaulay.  We give two corollaries of Theorem \ref{thm:glicci} concerning weakly geometrically vertex decomposable ideals:

\begin{corollary}
A geometrically vertex decomposable ideal is weakly geometrically vertex decomposable.
\end{corollary}
\begin{proof}

We will proceed by induction on $\dim(R)$.  Suppose that a geometrically vertex decomposable ideal $I$ has the geometric vertex decomposition $\text{in}_y I = C_{y,I} \cap (N_{y,I}+\langle y \rangle)$ with respect to some variable $y = x_j$ of $R$.  By Proposition \ref{prop:radical}, $I$ is radical, and so, if the geometric vertex decomposition is degenerate, then $I = N_{y,I}+\langle y \rangle$ or $I = N_{y,I}$ by Proposition \ref{prop:degen-rad}, and the result is immediate by induction. Hence, we may assume that the geometric vertex decomposition is nondegenerate.  From Theorem \ref{thm:glicci}, we know that $N_{y,I} \cap \kappa[x_1,\ldots,\widehat{y},\ldots,x_n]$ is glicci and so Cohen--Macaulay.  Hence, so is $N_{y,I}$.  Proposition \ref{prop:radical} tells us that $N_{y,I} \cap \kappa[x_1,\dots,\widehat{y},\dots, x_n]$ is radical.  Hence, so is $N_{y,I}$.  By induction, $C_{y,I} \cap \kappa[x_1,\dots,\widehat{y},\dots, x_n]$ is weakly geometrically vertex decomposable and so, by Observation \eqref{Observation1}, $C_{y,I}$ is weakly geometrically vertex decomposable, completing the proof.
\end{proof}

\begin{corollary}\label{weakGVDimpliesgliggi}
A weakly geometrically vertex decomposable ideal is both radical and glicci.
\end{corollary}
\begin{proof}
The proofs of Proposition \ref{prop:radical} and Theorem \ref{thm:glicci} easily adapt to the weakly geometrically vertex decomposable setting. 
In particular, in these proofs, we only used that the ideal $N_{y,I}$ was geometrically vertex decomposable to obtain that $N_{y,I}$ was Cohen--Macaulay, radical, saturated, and unmixed.  The first two of those properties are automatic from the definition of weakly geometrically vertex decomposable and the last two follow because Cohen--Macaulay ideals are always unmixed and always saturated unless they are the maximal ideal.
\end{proof}

\begin{remark}\label{rem:weakVert}
Let $\Delta$ be a simplicial complex on vertex set $[n]$. As with Proposition \ref{prop:VDiffGVD}, it is a straightforward exercise to show that $\Delta$ is weakly vertex decomposable in the sense of U. Nagel and T. R\"omer (see \cite[Definition 2.2]{NR08}) if and only if $I_\Delta$ is weakly geometrically vertex decomposable. Furthermore, by restricting our proofs of Corollary \ref{weakGVDimpliesgliggi} and Theorem \ref{thm:glicci} to the case of squarefree monomial ideals, we recover \cite[Theorem 3.3]{NR08}, which asserts that $I_\Delta$ is \emph{squarefree glicci} whenever $\Delta$ is a weakly vertex decomposable simplicial complex.  
\end{remark}

We end by showing that the condition of being weakly geometrically vertex decomposable is strictly weaker than that of being geometrically vertex decomposable.

\begin{example}
This example is a minor modification of Example \ref{ex:nolex}.  Take $I$ to be the ideal of $\kappa[x,y,z,w,r,s]$, generated by $\{y(zs-x^2), ywr, wr(x^2+s^2+z^2+wr)\}$, which, following the argument of Example \ref{ex:nolex}, is a universal Gr\"obner basis.
 Observe that $I$ is squarefree only in $y$, so we must first degenerate with respect to $y$, which yields $C_{y,I} = \langle zs-x^2,wr \rangle$ and $N_{y,I} = \langle (wr)(x^2+s^2+z^2+wr) \rangle$.  We saw in Example \ref{ex:nolex} that the contraction of $C_{y,I}$ to $\kappa[x,z,w,r,s]$ was geometrically vertex decomposable.  Here the contraction of $N_{y,I}$ to $\kappa[x,z.w,r,s]$ is clearly radical and Cohen--Macaulay but has no geometric vertex decomposition because it is not squarefree in any variable.  Hence, $I$ is weakly geometrically vertex decomposable but not geometrically vertex decomposable.
\end{example}

\subsection{Applications to Gr\"obner bases and degenerations}\label{sect:GBapplications}

One can not in general transfer the Cohen--Macaulay property from an ideal to its initial ideal or from one component of a variety to the whole variety.  However, in the context of geometric vertex decomposition, we can use the combination of Cohen--Macaulyness of a homogeneous ideal $I$ and of the component $N_{y,I}+\langle y \rangle$ (equivalently, of $N_{y,I}$) to infer the same about $\text{in}_y I$.

\begin{corollary}\label{cor:allCM}
Suppose that $\text{in}_y I = C_{y,I} \cap (N_{y,I}+\langle y \rangle)$ is a nondegenerate geometric vertex decomposition of the homogeneous ideal $I\subseteq R$ and that both $N_{y,I}$ and $I$ are Cohen--Macaulay. Then, $C_{y,I}$ and $\text{in}_y I$ are Cohen--Macaulay as well.  
\end{corollary}

\begin{proof}
For convenience, write $N = N_{y,I}$ and $C = C_{y,I}$.   Because $I$ and $N$ are Cohen--Macaulay, they are unmixed.  Hence, we may apply Theorem \ref{thm:onestep} to see that $I/N \cong C/N$.  It is easy to see that $C/N \xrightarrow{y} \text{in}_y I/N$ is also an isomorphism, and so $I/N \cong \text{in}_y I/N$.  Let $m$ denote the homogeneous maximal ideal of $R$, and let $H^i_m(M)$ denote the $i^{th}$ local cohomology module of the $R$-module $M$ with support in $m$.  Because $I$ homogeneous implies $\text{in}_y I$ homogeneous, it is sufficient to check Cohen--Macaulayness at $m$.  Let $d = \dim(R/I) = \dim(R/N)-1$.  The short exact sequence $0 \rightarrow I/N \rightarrow R/N \rightarrow R/I \rightarrow 0$ tells us that $H^i_m(I/N) \cong H^{i-1}_m(R/I) = 0$ for all $i \leq d$ because $H^i_m(R/N) = 0$ for all $i \leq d$.  Then from the short exact sequence $0 \rightarrow \text{in}_y I/N \rightarrow R/N \rightarrow R/\text{in}_y I \rightarrow 0$ together with the fact that $\text{in}_y I/N \cong I/N$, we have \[
H^{i-1}_m(R/\text{in}_y I) \cong H^i_m(\text{in}_y I/N) \cong H^i_m(I/N) \cong H^{i-1}_m(R/I) = 0
\] for all $i-1<d = \dim(R/\text{in}_y(I))$, and so $R/\text{in}_y(I)$ is Cohen--Macaulay.  

The argument in the case of $C$ follows the same line using the short exact sequence $0 \rightarrow C/N \rightarrow R/N \rightarrow R/C \rightarrow 0$.  
\end{proof}

One consequence of Corollary \ref{cor:allCM} is that we may omit as a hypothesis that $C_{y,I}$ be unmixed in Corollary \ref{cor:gvdToLia} whenever $I$ is Cohen--Macaulay.

We will now describe conditions that allow one to use the map constructed in Theorem \ref{thm:onestep} in order to conclude that a known set of generators for $I$ forms a Gr\"obner basis when Gr\"obner bases for $C_{y,I}$ and $N_{y,I}$ are known.  The result complements the framework of \cite{KMY09}, in which one begins with a Gr\"obner basis of $I$ and concludes that the resultant generating sets of $C_{y,I}$ and $N_{y,I}$ are also Gr\"obner bases.  For convenience, we recall a lemma from \cite{GMN13}: 

\begin{lemma}\label{GVLmodified}\cite[Lemma 1.12]{GMN13}
Fix a term order $<$ and homogeneous ideals $N$, $C$, $I$, and $\tilde{I}$ in a polynomial ring with $N \subseteq I \cap C$ and $\tilde{I} \subseteq \text{in}_<(I)$.  If $I/N \cong [C/N](-1)$ and $\tilde{I}/\text{in}_<(N) \cong [\text{in}_<(C)/\text{in}_<(N)](-1)$, then $\tilde{I} = \text{in}_<I$.  
\end{lemma}

Although the lemma is stated differently in \cite{GMN13}, the proof given there also applies to the conditions as stated above.

\begin{corollary}\label{cor:Gb(I)}
Let $I = \langle yq_1+r_1,\dots, yq_k+r_k,h_1,\dots, h_\ell \rangle$ be a homogenous ideal of  $R$ with $y = x_j$ some variable of $R$ and $y$ not dividing any term of any $q_i$ for $1 \leq i \leq k$ nor of any $h_j$ for $1 \leq j \leq \ell$. Fix a term order $<$, and suppose that $\mathcal{G}_C = \{q_1,\dots, q_k,h_1,\dots, h_\ell\}$ and $\mathcal{G}_N = \{h_1,\dots, h_\ell\}$ are Gr\"obner bases for the ideals they generate, which we call $C$ and $N$, respectively.  Assume that $\text{in}_<(yq_i+r_i) = y \cdot \text{in}_<q_i$ for all $1 \leq i \leq k$. Assume also that $\hgt(I)$, $\hgt(C)>\hgt(N)$ and that $N$ is unmixed.  Let $M = \begin{pmatrix}
q_1& \cdots & q_k\\
r_1& \cdots & r_k
\end{pmatrix}.$ If the ideal of $2$-minors of $M$ is contained in $N$, then the given generators of $I$ are a Gr\"obner basis.
\end{corollary}
\begin{proof}
Following the proof of Theorem \ref{thm:onestep}, the conditions that $\hgt(I)$, $\hgt(C)>\hgt(N)$ and that $N$ be unmixed imply that the elements $u$ and $v$ of Theorem \ref{thm:onestep} are non-zero-divisors on $R/N$.  The condition that the ideal of $2$-minors of $M$ be contained in $N$ implies that \[
r_i(a_1q_1 +\cdots+a_kq_k)-q_i(a_1r_1 +\cdots+a_kr_k)  = a_1(r_iq_1-r_1q_i)+\cdots+a_k(r_iq_k-r_kq_i)\in N
\] for every $1 \leq i \leq k$.  The remainder of the argument from Theorem \ref{thm:onestep} that $\overline{\varphi}: [C/N](-1) \rightarrow I/N$ is an isomorphism remains intact in this setting.

Set $\tilde{I} = \langle y \cdot \text{in}_<(q_1), \ldots, y \cdot \text{in}_<(q_k), \text{in}_<(h_1), \ldots, \text{in}_<(h_\ell) \rangle$.  Because $\mathcal{G}_C$ and $\mathcal{G}_N$ are Gr\"obner bases, we know  that $\text{in}_< C = \langle \text{in}_<(q_1), \ldots, \text{in}_<(q_k), \text{in}_<(h_1), \ldots, \text{in}_<(h_\ell) \rangle$ and that $\text{in}_< N = \langle \text{in}_<(h_1), \ldots, \text{in}_<(h_\ell)\rangle $. Because $\text{in}_< (yq_i+r_i) = y \cdot \text{in}_<(q_i)$ for each $1 \leq i \leq k$, the map $[\text{in}_<C/\text{in}_<N](-1)\xrightarrow{y} \tilde{I}/\text{in}_<N$ is also an isomorphism.  It follows from Lemma \ref{GVLmodified} that $\tilde{I} = \text{in}_< I$.
\end{proof}

\begin{example}[The Veronese Embedding]
As an application of Corollary \ref{cor:Gb(I)}, we give a concise inductive proof that the usual set of homogeneous equations defining the image of the $d^{\rm{th}}$ Veronese $\nu_d:\mathbb{P}^1 \rightarrow \mathbb{P}^d$ forms a Gr\"obner basis for any $d\geq 1$.  With homogeneous coordinates $[s:t]$ on $\mathbb{P}^1$ and $[x_0: \cdots :x_d]$ on $\mathbb{P}^d$, recall that the $d^{th}$ Veronese is the map $[s:t] \mapsto [s^d:s^{d-1}t:\cdots:st^{d-1}:t^d]$. 
Let $A_d = \begin{pmatrix}
x_0 & x_1 &\cdots & x_{d-1}\\
x_1 & x_2 &\cdots & x_d
\end{pmatrix}$, let $\mathcal{G}_d$ denote the set of $2\times 2$ minors of $A_d$, and let $I = \langle \mathcal{G}_d\rangle$ be the ideal generated by $\mathcal{G}
_d$.
The image of the $\nu_d$ is defined by $I$, 
which is to say that there is a ring isomorphism $\dfrac{\kappa[x_0, \ldots, x_d]}{I} \rightarrow \kappa[s^d, s^{d-1}t, \ldots, st^{d-1}, t^d]\subseteq \kappa[s,t]$ given by $x_i \mapsto s^{d-i}t^i$ for $0 \leq i \leq d$. 

We now show that $\mathcal{G}_d$ is a Gr\"obner basis of $I$ with respect to the lexicographic monomial order with $x_d>x_{d-1}>.\cdots>x_1>x_0$. We proceed by induction on $d$, noting that $d = 1$ is trivial because in that case $I = \langle 0 \rangle$.   For $d \geq 2$ and with notation as in Corollary \ref{cor:Gb(I)}, notice that $C= \langle x_0, \ldots, x_{d-2} \rangle$ and that $N = \langle \mathcal{G}_{d-1}\rangle$, whose given generators are a Gr\"obner basis by induction.  Because $N$ is a prime ideal properly contained in $C \cap I$, we know both that $N$ is unmixed and also that $\hgt(I)$, $\hgt(C)>\hgt(N)$.  Lastly, observe that the ideal generated by the $2\times 2$ minors of $M = \begin{pmatrix}
x_0& x_1 &\cdots & x_{d-2}\\
x_1x_{d-1}& x_2 x_{d-1} &\cdots & x_{d-1}^2
\end{pmatrix}$ is equal to  $x_{d-1} \cdot N$ and so is contained in $N$. Thus, the result follows from Corollary \ref{cor:Gb(I)}. \qedhere

\end{example}

\section{Some well-known families of ideals are glicci}\label{sect:applications}

Many well-known classes of ideals Gr\"obner degenerate to Stanley--Reisner ideals of vertex decomposable complexes. In this section, we recall a few of these classes and deduce that they are glicci, thus providing further evidence for an affirmative answer to the question of whether every homogeneous Cohen--Macaulay ideal is glicci \cite[Question 1.6]{KM+01}. As in Section \ref{sect:GVDGlicci}, we will assume throughout this section that the field $\kappa$ is infinite.

The main result we need for our applications is as follows. It is immediately obtained by combining Proposition \ref{prop:simplicial} with Corollary \ref{weakGVDimpliesgliggi}. 

\begin{corollary}\label{cor:AutomaticallyGlicci}
Let $I\subseteq \kappa[x_1,\dots, x_n]$ be a homogeneous ideal, and let $<$ denote the lexicographic order with $x_1>x_2>\cdots >x_n$. If $\text{in}_<I$ is the Stanley--Reisner ideal of a $<$-compatibly vertex decomposable simplicial complex on $[n]$ for the vertex order $1>2>\cdots >n$, then $I$ is glicci.
\end{corollary}

We now discuss three classes of ideals which satisfy the hypotheses of Corollary \ref{cor:AutomaticallyGlicci}. We omit many definitions of the particular ideals in question, and instead provide references. 

\subsubsection*{Schubert determinantal ideals}
Let $X = (x_{ij})$ be an $n\times n$ matrix of variables and let $R = \kappa[x_{ij}]$ be the polynomial ring in the matrix entries of $X$.  
Given a permutation $w\in S_n$, there is an associated generalized determinantal ideal $I_w\subseteq R$, called a \emph{Schubert determinantal ideal}. Schubert determinantal ideals and their corresponding \emph{matrix Schubert varieties} were introduced by W. Fulton in \cite{Fulton}.

Fix the lexicographical monomial order $<$ on $R$ defined by $x_{ij}>x_{kl}$ if $i<k$ or $i = k$ and $j>l$. This monomial order is \emph{antidiagonal}, that is, the initial term of the determinant of a submatrix $Y$ of $X$ is the product of the entries along the antidiagonal of $Y$. For this monomial order, $\text{in}_<I_w$ is the Stanley--Reisner ideal of a simplicial complex, called a \emph{subword complex}, which is $<$-compatibly vertex decomposable (see \cite{KnutsonMiller} or \cite[Ch. 16.5]{MillerSturmfels}). Corollary \ref{cor:AutomaticallyGlicci} thus immediately implies: 

\begin{proposition}\label{prop:SchubGlicci}
Schubert determinantal ideals are glicci.
\end{proposition}

\subsubsection*{Graded lower bound cluster algebras}
Cluster algebras are a class of combinatorially-defined commutative algebras that were introduced by S. Fomin and A. Zelevinsky at the turn of the century \cite{FZI}. \emph{Lower bound algebras}, introduced in \cite{BFZ} are related objects: each lower bound algebra is contained in an associated cluster algebra, and this containment is equality in certain cases (i.e. in the \emph{acyclic} setting, see \cite[Theorem 1.20]{BFZ}).

Each (skew-symmetric) lower bound algebra is defined from a quiver. Indeed, given a quiver $Q$, there is an associated polynomial ring $R_Q = \kappa[x_1,\dots, x_n,y_1,\dots, y_n]$ and ideal $K_Q\subseteq R_Q$ such that the lower bound algebra $\mathcal{L}_Q$ associated to $Q$ can be expressed as $\mathcal{L}_Q = R_Q/K_Q$. Fix the lexicographical monomial order with $y_1>\cdots>y_n>x_1>\cdots >x_n$. By \cite[Theorem 1.7]{MRZ} and the proof of \cite[Theorem 3.3]{MRZ}, $\text{in}_<K_Q$ is the Stanley--Reisner ideal of a simplicial complex $\Delta$ on the vertex set $\{y_1,\dots, y_n,x_1,\dots, x_n\}$, which has vertex decomposition compatible with $<$. Consequently, by Proposition \ref{prop:simplicial}, we have the following:

\begin{proposition}
The ideal $K_Q$ is geometrically vertex decomposable. When $K_Q$ is homogeneous, it is glicci.
\end{proposition}

\begin{remark}
It follows from \cite[Theorem 1.7]{MRZ} that $K_Q$ is homogeneous if and only if $Q$ has no \emph{frozen vertices} and $Q$ has exactly two arrows entering each vertex and two arrows exiting each vertex.
\end{remark}

\subsubsection*{Ideals defining equioriented type $A$ quiver loci}
Let $d_0, d_1, \dots, d_n$ be a sequence of positive integers and consider the product of matrix spaces $\texttt{Hom}$, and product of general linear group $\texttt{GL}$ defined as follows:
\[
\texttt{Hom}:=\oplus_{i = 1}^n \text{Mat}_{d_{i-1}\times d_i}(\kappa), \quad \texttt{GL}:=\oplus_{i = 0}^n \text{GL}_{d_i}(\kappa).
\]
The group $\texttt{GL}$ acts on $\texttt{Hom}$ on the right by conjugation: $(M_i)_{i=1}^n\bullet (g_i)_{i=0}^n = (g_{i-1}^{-1}M_i g_i)_{i=1}^n$. Closures of $\texttt{GL}$-orbits are called \textbf{equioriented type $A$ quiver loci}. \emph{Buchsbaum-Eisenbud varieties of complexes} are special cases of these quiver loci. An introduction to equioriented type $A$ quiver loci and related combinatorics can be found in \cite[Ch. 17]{MillerSturmfels}. 

\begin{proposition}
Equioriented type $A$ quiver loci are glicci. In particular, varieties of complexes are glicci. 
\end{proposition}

\begin{proof} 
Let $\Omega\subseteq \texttt{Hom}$ be an equioriented type $A$ quiver locus, and let $I(\Omega)$ be its (homogeneous and prime) defining ideal in the polynomial ring $\kappa[\texttt{Hom}]$. It follows from results of A. Zelevinsky \cite{Zelevinsky} and V. Lakshmibai and P. Magyar \cite{LakshmibaiMagyar} that there
is a polynomial ring $R$ with $\kappa[\texttt{Hom}]\subseteq R$, a \emph{Kazhdan-Lusztig ideal} $J\subseteq R$, and an ideal $L$ generated by the indeterminates in $R\setminus \kappa[\texttt{Hom}]$ such that $J = I(\Omega)R+L$. (Here $I(\Omega)R$ denotes the extension of the ideal $I(\Omega)$ to $R$.) As shown in \cite{WooYongGrobner}, each Kazhdan-Lusztig ideal Gr\"obner degenerates to the Stanley--Reisner ideal of a subword complex, and this degeneration is compatible with the vertex decomposition of the complex. Consequently, $J$ is geometrically vertex decomposable. Thus $I(\Omega)$ is geometrically vertex decomposable, hence glicci. 
\end{proof}

\section{From G-biliaisons to geometric vertex decompositions}\label{sect:glicciGVD}
In this section, we give something of a converse to Theorem \ref{thm:onestep}. In that theorem, we showed that, under mild assumptions, a geometric vertex decomposition gives rise to an elementary $G$-biliaison and showed that the isomorphism of that elementary $G$-biliaison has a very particular form.  
In this section, we show that every elementary $G$-biliaison in which the isomorphism has the same form as the ones constructed in Theorem \ref{thm:onestep} gives rise to a geometric vertex decomposition.  The precise statement of the main theorem of this section is below. 
As usual, throughout this section we will let $R$ denote the polynomial ring $\kappa[x_1,\ldots,x_n]$.  

\begin{theorem}\label{thm:linkToGVD}
Let $I$, $C$, and $N \subseteq I \cap C$ be ideals of $R$, and let $<$ be a $y$-compatible term order.  Suppose that $I$ is squarefree in $y$ and that no term of any element of the reduced Gr\"obner basis of $N$ is divisible by $y$.  Suppose further that there exists an isomorphism $\phi: C/N \xrightarrow{f/g} I/N$ of $R/N$-modules for some $ f, g \in R$ not zero-divisors on $R/N$, and $\text{in}_y(f)/g = y$.  Then $\text{in}_y I = C \cap ( N+\langle y \rangle )$ is a geometric vertex decomposition of $I$.  
\end{theorem}

\begin{proof}
Recall that $I$ must have a Gr\"obner basis of the form $\{yq_1+r_1, \ldots, yq_k+r_k, h_1, \ldots, h_\ell\}$ where $y$ does not divide any term of any $q_i$ or $r_i$ for any $1 \leq i \leq k$ nor any $h_j$ for any $1 \leq j \leq \ell$ because $I$ is squarefree in $y$.  Hence, $\text{in}_y I = \langle yq_1, \ldots, yq_k, h_1, \ldots, h_\ell \rangle$, and this generating set is a Gr\"obner basis of $\text{in}_y I$.

We claim first that $N = \langle h_1, \ldots, h_\ell \rangle$.  Because no term of any element of the reduced Gr\"obner basis of $N \subseteq I$ is divisible by $y$, each such element must be a polynomial in the $h_i$ for $1 \leq i \leq \ell$, and so $N \subseteq \langle h_1, \ldots, h_\ell \rangle$, from which it follows that $\text{in}_y(N) = N$ and that $y$ is not a zero-divisor on $R/N$.  Conversely, suppose there is some $h_i \in I \setminus N$ for some $1 \leq i \leq \ell$.  Then there exists some $c \in C \setminus N$ and $n \in N$ such that $c f = gh_i+n$, where $c$ has been chosen to have the smallest possible $d$ for which $y^d$ divides $\text{in}_y(c)$. Taking initial $y$-forms yields
\[
\text{in}_y(c)yg = \text{in}_y(c)\cdot \text{in}_y(f)  =\text{in}_y(c f)  = \text{in}_y(gh_i+n).
\]
Note that $\text{in}_y g = g$ by the assumption that $g$ divides $\text{in}_y f$ and that, if $g \in \langle y^m \rangle \setminus \langle y^{m+1} \rangle$, then $gh_i \in \langle y^m \rangle \setminus \langle y^{m+1} \rangle$ while $\text{in}_y(c)yg  \in \langle y^{m+1} \rangle$. Thus, we see that $\text{in}_y(c)yg = \text{in}_y(n)$. Hence, $\text{in}_y(c)yg\in N$ (as $\text{in}_y N = N$) and so $\text{in}_y(c)\in N$ (as $N:\langle yg\rangle = N$).

Set $c' = c-\text{in}_y(c) \in C \setminus N$, which has an initial $y$-form not divisible by $y^d$.  But $\varphi(c'+N) = \varphi(c+N) = h_i+N$, contradicting minimality of $d$.  Hence, $N = \langle h_1, \ldots, h_\ell \rangle$.  

Next, we claim that $C = \langle q_1, \ldots, q_k, h_1, \ldots, h_\ell \rangle$.  By assumption, $N \subseteq C$, and so it suffices to show that the $q_i$ for $1 \leq i \leq k$ generate $C$ over $N$.  In order to establish this, we will show that $\psi: C/N \xrightarrow{y} \text{in}_y(I)/N$ is an isomorphism.  For each $1 \leq i \leq k$, let $c_i \in C \setminus N$ be a representative of the preimage under $\phi$ of the class of $yq_i+r_i$ in $I/N$ with the smallest possible $d_i$ so that $\text{in}_y(c_i) \notin \langle y^{d_i} \rangle$.  We will show that the image under $\psi$ of the class of $c_i$ is the class of $yq_i$.

First, we will show that $d_i = 1$ for all $1 \leq i \leq k$.  Suppose, for contradiction, that some $d_i > 1$, i.e., that some $\text{in}_y(c_i) \in \langle y \rangle$.  Then $\text{in}_y(g(yq_i+r_i)+n_i) = y \cdot \text{in}_y(fc_i)\in \langle y^2\rangle$.  Because neither $q_i$ nor $r_i$ has any term divisible by $y$, we must have \[
y g \text{in}_y(c_i) = \text{in}_y\left(g(yq_i+r_i)+n_i\right) = \text{in}_y(n_i) \in \text{in}_y (N) = N,
\] and so $\text{in}_y(c_i) \in N$.  Then $c_i' = c_i-\text{in}_y(c_i) \in C\setminus N$ and $\text{in}_y(c_i') \notin \langle y^{d_i} \rangle$.  But still $c_i'$ represents a preimage of $yq_i+r_i$, contradicting minimality of $d_i$.  Hence, $\text{in}_y(c_i) \notin \langle y \rangle$, which establishes that $\text{in}_y(c_i) = c_i$.

From the former fact and the relationship $y g \text{in}_y(c_i) = \text{in}_y(g(yq_i+r_i)+n_i)$, we have either $y g c_i = gyq_i$ (if $\text{in}_y(g(yq_i+r_i)+n_i) = \text{in}_y(g(yq_i+r_i))$) or $y g c_i = gyq_i+gyn_i'$ for some nonzero $n_i' \in N$ (using $N:\langle yg\rangle = N$ and $\text{in}_y(N) = N$). In either case, the $\psi(c_i)$ is the class of $yq_i$ in $\text{in}_y(I)/N$, which is to say that $\psi$ is surjective.  Also, $\psi$ is injective because $y$ is not a zero-divisor on $R/N$.  Now because the $yq_i$ for $1 \leq i \leq k$ generate $\text{in}_y(I)$ over $N$ and $\psi$ is an isomorphism under which the preimage of the class of $yq_i$ is the class of $q_i$ for each $1 \leq i \leq k$, it must be that $C$ is generated over $N$ by $\{q_1, \ldots, q_k\}$ and that $C = \langle q_1, \ldots, q_k, h_1, \ldots, h_\ell \rangle$.

By \cite[Theorem 2.1(a)] {KMY09}, the specified generating sets for $\text{in}_y(I)$, $N$, and $C$ are all Gr\"obner bases for them, and so it follows from \cite[Theorem 2.1(b)]{KMY09} that $\text{in}_y(I) = C \cap (N+\langle y \rangle)$ is a geometric vertex decomposition of $I$.  
\end{proof}

\begin{example}\label{ex:standard}
To illustrate this correspondence between elementary $G$-biliaison and geometric vertex decomposition, we consider a classical example.  If $I$ is the ideal of $2$-minors of the matrix $M = \begin{pmatrix}
x_{11} & x_{12} & x_{13} \\
x_{21} & x_{22} & x_{23} 
\end{pmatrix}$, $C = \langle x_{11}, x_{12} \rangle$, $N = \langle x_{22}x_{11}-x_{21}x_{12} \rangle$, $f = x_{23}x_{12}-x_{22}x_{13}$, and $g = x_{12}$ in $\kappa[x_{11}, \ldots, x_{23}]$, then the multiplication by $f/g$  map $[C/N](-1) \xrightarrow{f/g} I/N$ gives an elementary $G$-biliaison.  Using any lexicographic order with $x_{23}$ largest, we take $C = C_{x_{23},I}$ and $N = N_{x_{23},I}$, and then $\text{in}_{x_{23}}(I) = C \cap (N+\langle x_{23} \rangle)$ is a geometric vertex decomposition.
\end{example}

Notice that in Theorem \ref{thm:linkToGVD} we use only the isomorphism that makes up an elementary $G$-biliaison to construct a geometric vertex decomposition in the sense of \cite{KMY09}.  In this direction, we do not need to assume that the ideals $I$, $C$, and $N$ are homogeneous or saturated or even unmixed, nor that $N$ is Cohen--Macaulay or $G_0$.  Of course, the isomorphism $\phi$ increases degree by $\deg(y)$ whenever that makes sense.

\begin{remark}\label{rmk:invertible}
In the notation and under the hypotheses of Theorem \ref{thm:onestep}, the construction in Theorem \ref{thm:onestep} produces an isomorphism $I/N_{y,I} \xrightarrow{v/u} C_{y,I}/N_{y,I}$ with $\dfrac{\text{in}_{y}(v)}{u} = y$.  In particular, the hypotheses of Theorem \ref{thm:linkToGVD} are satisfied.  It is not hard to see that the geometric vertex decomposition produced by Theorem \ref{thm:linkToGVD} is the same one assumed before applying Theorem \ref{thm:onestep}.

If we begin, instead, with an isomorphism between $I/N$ and $C/N$ and accompanying hypotheses of Theorem \ref{thm:linkToGVD}, we may first apply Theorem \ref{thm:linkToGVD} to obtain a geometric vertex decomposition satisfying the hypotheses of Theorem \ref{thm:onestep}.  If we then apply the construction in Theorem \ref{thm:onestep}, we obtain an isomorphism between $I/N$ and $C/N$, but it need not be the same isomorphism we began with.   For example, we may begin with the multiplication by $f/g$ map from Example \ref{ex:standard} but produce the multiplication by $v/u$ map for $v = a_1f+a_2(x_{23}x_{11}-x_{21}x_{13})$ and $u = a_1x_{12}+a_2x_{11}$ for a generic choice of scalars $a_1$ and $a_2$.
\end{remark}

One has to be quite careful in tracking the correspondence between a particular biliaison and a geometric vertex decomposition.  In particular, somewhat surprisingly, the condition that the reduced Gr\"obner basis of $N$ has no term divisible by $y$ cannot be discarded while preserving the canonical mapping noted in Remark \ref{rmk:invertible}.  For example, we consider a modification of Example \ref{ex:standard} by letting $I'  = I+\langle x_{23}x_{10}-x_{13}x_{20} \rangle$, $N' = N+\langle x_{23}x_{10}-x_{13}x_{20} \rangle$, and $C' = C+\langle x_{23}x_{10}-x_{13}x_{20} \rangle$.  We think of this example as naturally occurring from the matrix $M' = \begin{pmatrix}  
x_{10} & x_{11} & x_{12} & x_{13} \\
x_{20} & x_{21} & x_{22} & x_{23} 
\end{pmatrix}$, from which the ideal $I'$ is generated by all $2$-minors involving any 2 of the last 3 columns or exactly the first and fourth columns.  Here, taking $f' = x_{23}x_{12} - x_{22}x_{13} \in I'\setminus N'$ and $g' = x_{12} \in C' \setminus N'$ yields an isomorphism $C'/N' \xrightarrow{f'/g'} I'/N'$.  Taking lexicographic order with respect to $x_{23}>x_{13}>x_{22}>\cdots>x_{10}$ and noting that $N'$ is prime, it is not hard to check that the hypotheses of Theorem \ref{thm:linkToGVD} are satisfied aside from the hypothesis that the reduced Gr\"obner basis of $N'$ have no term divisible by $x_{23}$.  However, the geometric vertex decomposition of $I'$ with respect to $x_{23}$ is \begin{align*}
&\langle x_{21}x_{13}x_{10}-x_{20}x_{13}x_{11}, x_{22}x_{11}-x_{21}x_{12}, x_{22}x_{13}x_{10}-x_{20}x_{13}x_{12}, x_{23}x_{10}, x_{23}x_{11}, x_{23}x_{12} \rangle= \\
&\langle x_{10}, x_{11}, x_{12} \rangle \cap (  \langle x_{21}x_{13}x_{10}-x_{20}x_{13}x_{11}, x_{22}x_{11}-x_{21}x_{12}, x_{22}x_{13}x_{10}-x_{20}x_{13}x_{12}\rangle+\langle x_{23} \rangle ).
\end{align*} In particular, $\text{in}_{x_{23}}  I'  \neq C' \cap (N'+\langle x_{23} \rangle )$.  

In the other direction, the elementary $G$-biliaison constructed from Theorem \ref{thm:onestep} yields the isomorphism $C'/\tilde{N} \xrightarrow{f'/g'} I'/\tilde{N}$ for \[
\tilde{N} = (x_{21}x_{13}x_{10}-x_{20}x_{13}x_{11}, x_{22}x_{11}-x_{21}x_{12}, x_{22}x_{13}x_{10}-x_{20}x_{13}x_{12}),
\] which is not the same elementary $G$-biliaison we began with.

Remark \ref{rmk:invertible} gives rise to the question of whether or not there is a sort of moving lemma applicable to this situation that would allow us to replace the module $N$ with a Cohen--Macaulay and $G_0$ module $\tilde{N}$ that also links $C$ to $I$ but does not involve $y$.  More precisely:

\begin{question} With notation as in Theorem \ref{thm:linkToGVD}, suppose that $I$ is squarefree in $y$ and that there exists an elementary $G$-biliaison given by the isomorphism $\phi: C/N \xrightarrow{f/g} I/N$ of $R/N$-modules for some $ f \in I$, $g \in C$, and $\text{in}_y(f)/g = y$.  Do not assume that the reduced Gr\"obner basis of $N$ does not involve $y$.  From \cite[Theorem 2.1(b)]{KMY09}, $I$ must have some geometric vertex decomposition with respect to $y$.  If $\text{in}_y(I) = \tilde{C} \cap (\tilde{N}+\langle y \rangle )$ is a geometric vertex decomposition of $I$, then Theorem \ref{thm:onestep} requires that there be an isomorphism $\tilde{C}/\tilde{N} \rightarrow I/\tilde{N}$.  In particular, though, will multiplication by $f/g$ always be an isomorphism from $C/\tilde{N}$ to $I/\tilde{N}$?  Need $\tilde{N}$ be Cohen--Macaulay and $G_0$?
\end{question}

\section{The mixed case and sequential Cohen--Macaulayness}\label{sect:nonpure}

A nonpure version of vertex decomposition was introduced in \cite{BW97}, in which the authors study non-pure shellable complexes, including their homotopy types and combinatorially significant direct sum decompositions of their Stanley--Reisner rings.  It has been shown that if a simplicial complex is non-pure vertex decomposable, then it is non-pure shellable \cite[Theorem 11.3]{BW97}.  And it is not hard to see that a non-pure shellable simplicial complex is sequentially Cohen--Macaulay (i.e., its associated Stanley--Reisner ring is sequentially Cohen--Macaulay).  For background on sequential Cohen--Macaulayness, introduced by Stanley, we refer the reader to \cite[Section III.2]{Sta96}.  This story parallels the well-known history of the pure case, which is summarized in Section \ref{sect:GVD}.  This non-pure version has been applied particularly effectively in the study of edge ideals (see \cite{FVT07, VTH08, FH08, Woo09}).

In this section, we compare non-pure vertex decomposition with geometric vertex decomposition when $I$ is not necessarily unmixed, and we describe how geometric vertex decomposition can transfer the structure of sequential Cohen--Macaulayness in a manner similar to how $G$-biliaison transfers Cohen--Macaulayness in the unmixed case.  This result is stated precisely as Theorem \ref{thm:SCM}.  Throughout this section, we will assume that $\kappa$ is infinite, and we will let $R = \kappa[x_1,\ldots, x_n]$ with the standard grading.

We begin with the definition of a vertex decomposable complex when the complex not necessarily pure.

\begin{definition}\label{def:nonpurevd}\cite[Definition 11.1]{BW97}
A simplicial complex $\Delta$ is \textbf{vertex decomposable} if \begin{enumerate}
\item $\Delta$ is a simplex or $\Delta = \{\emptyset\}$, or 
\item there exists a vertex $v$ of $\Delta$ such that \begin{enumerate}
\item $\text{del}_\Delta(v)$ and $\text{lk}_\Delta(v)$ are vertex-decomposable and
\item no facet of $\text{lk}_\Delta(v)$ is a facet of $\text{del}_\Delta(v)$. 
\end{enumerate}
\end{enumerate}
A vertex $v$ as in condition (2) is called a \textbf{shedding vertex}.
\end{definition}

Let $\Delta$ be a simplicial complex on $[n]$ and $I_\Delta\subseteq R$ its Stanley--Reisner ideal. While any variable $y\in R\setminus I_\Delta$ that divides a minimal generator of $I_\Delta$ gives rise to a nondegenerate geometric vertex decomposition of $I_\Delta$ (see Definition \ref{def:gvdKMS}), $y$ need not correspond to a shedding vertex of $\Delta$. For example, if $I = (xy,xz)$, then $I = \text{in}_y I = C_{y,I} \cap (N_{y,I} +\langle y \rangle) = \langle x \rangle \cap \langle xz, y \rangle$ would be a geometric vertex decomposition, but $y$ is not a shedding vertex of $\Delta = \{\{x\}, \{y,z\}\}$ because $\{z\}$ is a facet of $\text{del}_\Delta(y) = \{\{z\}, \{x\}\}$ that is also a facet of $\text{lk}_\Delta(y)= \{\{z\}\}$.   In order to prevent an ideal from being geometrically vertex decomposable via nondegenerate geometric vertex decompositions at variables that do not correspond to shedding vertices, we propose an alternative definition of geometric vertex decomposition:

\begin{altdefinition}\label{def:nonpureGVD}
If \begin{enumerate} 
\item $\text{in}_y I = C_{y,I} \cap (N_{y,I}+\langle y \rangle)$, and
\item either $\sqrt{C_{y,I}} = \sqrt{N_{y,I}}$ or no minimal prime of $C_{y,I}$ is a minimal prime of $N_{y,I}$,
\end{enumerate} 
then we say that $\text{in}_y I = C_{y,I} \cap (N_{y,I}+\langle y \rangle)$ is a \textbf{geometric vertex decomposition of $I$ with respect to $y$}.
\end{altdefinition}

As in the original definition, we will call this geometric vertex decomposition \textbf{nondegenerate} if $C_{y,I} \neq \langle 1 \rangle$ and if $\sqrt{C_{y,I}} \neq \sqrt{N_{y,I}}$.  

In the unmixed case,  if $I$ is geometrically vertex decomposable and one step in that decomposition is a nondegenerate geometric vertex decomposition with respect to $y$, it is automatic that the minimal primes of $N_{y,I}$ and $C_{y,I}$ must be disjoint because the minimal primes of the former must all have height one less than those of the latter in virtue of Lemma \ref{lem:height} and unmixedness of $N_{y,I}$ and $C_{y,I}$.  Reinterpreting Definition \ref{def:geometricallyVertexDec} in terms of Definition \ref{def:nonpureGVD},  it is a straightforward exercise to see that a squarefree monomial ideal is geometrically vertex decomposable exactly when its Stanley--Reisner complex is vertex decomposable in the sense of Definition \ref{def:nonpurevd}.


We will now describe how geometric vertex decomposition behaves somewhat analogously to $G$-biliaison in the not necessarily unmixed case.  In particular, we will show in Theorem \ref{thm:SCM} that if $I$ is homogeneous and $R/N_{y,I}$ is Cohen-Macaulay, then $R/I$ is sequentially Cohen-Macaulay if and only if $R/C_{y,I}$ is sequentially Cohen-Macaulay. 
Just as in $G$-biliaison, in which $R/N_{y,I}$ is required to be not only Cohen--Macaulay but also $G_0$ in order to transfer the Cohen-Macaulay property between $R/C_{y,I}$ and $R/I$, we impose a stricter requirement on $R/N_{y,I}$ in Theorem \ref{thm:SCM} than the property we hope to pass between $R/I$ and $R/C_{y,I}$.  As in the unmixed case, we begin with a lemma concerning the heights of the ideals involved: 

\begin{lemma}\label{lem:nonpureheight}
If $I \subseteq R$ is a homogeneous ideal with nondegenerate geometric vertex decomposition $\text{in}_y I = C_{y,I} \cap (N_{y,I} + \langle y \rangle)$ in the sense of Definition \ref{def:nonpureGVD}, then $\hgt(I) = \hgt(N_{y,I})+1 \leq \hgt(C_{y,I})$.  
\end{lemma}
\begin{proof}
Because $\hgt(I) = \hgt(\text{in}_y I)$, it suffices to show that $\hgt(\text{in}_y I) = \hgt(N_{y,I})+1$.  Because every prime containing $\text{in}_y I$ must contain either $C_{y,I}$ or $N_{y,I}+\langle y \rangle$, we must have \[
\hgt(\text{in}_y I) = \min\{\hgt(C_{y,I}), \hgt(N_{y,I}+\langle y \rangle)\} =  \min\{\hgt(C_{y,I}), \hgt(N_{y,I})+1\}.
\]  Suppose $\hgt(C_{y,I})<\hgt(N_{y,I})+1$.  Then, because $N_{y,I} \subseteq C_{y,I}$, we must have $\hgt(N_{y,I}) = \hgt(C_{y,I})$.  Fix $P \in \text{Min}(C_{y,I})$ with $\hgt(P) = \hgt(C_{y,I})$.  Then $N_{y,I} \subseteq P$, and there cannot be a prime $Q \subsetneq P$ with $Q \in \text{Ass}(N_{y,I})$ or else we would have $\hgt(N_{y,I})<\hgt(C_{y,I})$, so $P \in \text{Min}(N_{y,I})$ as well, contradicting condition (2) of Definition \ref{def:nonpureGVD}.  Hence, we must have \[
\hgt(I) = \hgt(\text{in}_y I) = \hgt(N_{y,I})+1 \leq \hgt(C_{y,I}).\qedhere
\]
\end{proof}

Unlike in the unmixed case, we cannot hope to give an upper bound on the height of $C_{y,I}$ in terms of the heights of $I$ and $N_{y,I}$.  For example, if $I = (yx_1, \ldots, yx_d)$ for any $d \geq 1$, then $\hgt(C_{y,I}) = \hgt(\langle x_1, \ldots, x_d \rangle) = d$ while $\hgt(I) = 1 = \hgt( 0+\langle y \rangle) = \hgt(N_{y,I}+\langle y \rangle)$.

\begin{lemma}\label{lem:nonpureonestep}
Suppose that $I$ is a homogeneous ideal of $R$ and that $I$ possesses a nondegenerate geometric vertex decomposition (in the sense of Definition \ref{def:nonpureGVD}) with respect to a variable $y = x_j$ of $R$.  If $N_{y,I} $ has no embedded primes, then there is an isomorphism $I/N_{y,I} \cong [C_{y,I}/N_{y,I}](-1)$ as graded $R/N_{y,I}$-modules.  
\end{lemma}

\begin{proof}
We will modify the proof of Theorem \ref{thm:onestep}.  As there, we have a reduced Gr\"obner basis $\{yq_1+r_1,\dots, yq_k+r_k,h_1,\dots, h_\ell\}$ for $I$, and we let $C = C_{y,I} = \langle q_1,\dots, q_k, h_1,\dots, h_\ell\rangle$ and $N = N_{y,I} = \langle h_1,\dots, h_\ell \rangle$. The modification in the argument comes in the steps showing that neither $C$ nor $I$ is contained in any minimal prime of $N$.   Suppose first that $\langle q_1, \ldots, q_k \rangle \subseteq Q$ for some minimal prime $Q$ of $N$.  Then also $C \subseteq Q$.  Because $N \subseteq C$ and $Q$ is minimal over $N$, $Q$ must also be minimal over $C$, and so $N$ and $C$ share a minimal prime, in violation of Definition \ref{def:nonpureGVD}.  

Similarly, suppose $\langle yq_1+r_1, \ldots, yq_k+r_k \rangle \subseteq Q'$ for some minimal prime $Q'$ of $N$.  Then $I \subseteq Q'$.  But $N$ has a generating set that does not involve $y$, and so its minimal primes may be viewed as ideals of the ring $\kappa[x_1, \ldots, \hat{y}, \dots, x_n]$.  Then $I \subseteq Q'$ implies that each $yq_i \in Q'$, hence each $q_i \in Q'$.  But then again $N$ and $C$ share a minimal prime, in violation of Definition \ref{def:nonpureGVD}.  

Hence, neither $\langle q_1, \ldots, q_k \rangle$ nor $\langle yq_1+r_1, \ldots, yq_k+r_k \rangle$ is contained in any minimal prime of $N$.  Because $N$ has no embedded primes, it follows that neither $\langle q_1, \ldots, q_k \rangle$ nor $\langle yq_1+r_1, \ldots, yq_k+r_k \rangle$ is contained in any associated prime of $N$.  We now follow the remainder of the argument of Theorem \ref{thm:onestep}.
\end{proof}

\begin{theorem}{\label{thm:SCM}}
Let $I \subseteq R$ be a homogeneous ideal and $\text{in}_y I = C_{y,I} \cap (N_{y,I} + \langle y \rangle)$ a geometric vertex decomposition (in the sense of Definition \ref{def:nonpureGVD}).  If $R/N_{y,I}$ is Cohen--Macaulay, then $R/I$ is sequentially Cohen--Macaulay if and only if $R/C_{y,I}$ is.  
\end{theorem}

\begin{proof}
Let $N = N_{y,I}$ and $C = C_{y,I}$, both of which are homogeneous because $I$ is.  Because the graded $R$-submodules of $R/I$ (respectively, $R/C$) are the same as the graded $R/N$-submodules of $R/I$ (respectively, $R/C$), it suffices to show that $R/I$ is sequentially Cohen--Macaulay as an $R/N$-module if and only if $R/C$ is.  Let $S$ denote $R/N$ and $m$ the homogeneous maximal ideal of $S$.  Set $d = \dim(S)$.  Let $\omega_S$ be the canonical module of $S$ and $M^\vee$ the $S$-Matlis dual of a finitely generated graded $S$-module $M$.  By \cite[Theorem 1.4]{HS02}, it suffices to show that $H^i_m(R/I)^\vee = 0$ or $H^i_m(R/I)^\vee$ is Cohen--Macaulay of dimension $i$ for all $0 \leq i \leq \dim(R/I)$ if and only if $H^i_m(R/C)^\vee = 0$ or $H^i_m(R/C)^\vee$ is Cohen--Macaulay of dimension $i$ for all $0 \leq i \leq \dim(R/C)$.

We consider the long exact sequences of local cohomology corresponding to the short exact sequences \[
0 \rightarrow I/N \rightarrow S \rightarrow R/I \rightarrow 0
\] and  \[
0 \rightarrow C/N \rightarrow S \rightarrow R/C \rightarrow 0.
\] Now $H^i_m(S) = 0$ for all $i<d$ because $S$ is Cohen--Macaulay.  According to Lemma \ref{lem:nonpureonestep}, there is an isomorphism $I/N \cong C/N$.  Hence, \[
H^{i-1}_m(R/I) \cong H^i_m(I/N) \cong H^i_m(C/N) \cong H^{i-1}_m(R/C)
\] for all $i<d$.  

Hence, $H^i_m(R/C)^\vee$ and $H^i_m(R/I)^\vee$ are zero or nonzero alike and Cohen--Macaulay of dimension $i$ or not Cohen--Macaulay of dimension $i$ alike for all $0 \leq i \leq d-2$.  By Lemma \ref{lem:nonpureheight}, $\dim(R/C) \leq \dim(R/I) = d-1$.  For all $i>\dim(R/C)$, we know $H^i_m(R/C) = 0$.  Hence, it only remains to show that $H^{d-1}_m(R/I)^\vee$ is either $0$ or Cohen--Macaulay of dimension $d-1$ if and only if $H^{d-1}_m(R/C)^\vee$ is either $0$ or Cohen--Macaulay of dimension $d-1$.  Because $H^{d-1}_m(R/I)^\vee$ is a Noetherian $R/I$-module and $H^{d-1}_m(R/C)^\vee$ a Noetherian $R/C$-module, both have dimension at most $d-1$, and so it is enough to show that $H^i_m(H^{d-1}_m(R/I)^\vee) = 0$ for all $i<d-1$ if and only if $H^i_m(H^{d-1}_m(R/C)^\vee) = 0$ for all $i<d-1$.  We consider the short exact sequences
 \[
0 \rightarrow H^{d-1}_m(R/I) \rightarrow H^d_m(I/N) \rightarrow H^d_m(S) \rightarrow 0
\] and \[
0 \rightarrow H^{d-1}_m(R/C) \rightarrow H^d_m(C/N) \rightarrow H^d_m(S) \rightarrow 0.
\]  By graded local duality over $S$ (see \cite[Theorem 3.6.19]{BH93}), we have \[
0 \rightarrow \omega_S \rightarrow H^d_m(I/N)^\vee \rightarrow H^{d-1}_m(R/I)^\vee \rightarrow 0
\] and \[
0 \rightarrow \omega_S \rightarrow H^d_m(C/N)^\vee \rightarrow H^{d-1}_m(R/C)^\vee \rightarrow 0.
\]  Recalling that $\omega_S$ is a Cohen--Macaulay module of dimension $d$, we have $H^i_m(\omega_S) = 0$ for all $i \neq d$, and so \[
H^i_m(H^{d-1}_m(R/I)^\vee) \cong H^i_m(H^d_m(I/N)^\vee) \cong H^i_m(H^d_m(C/N)^\vee) \cong H^i_m(H^{d-1}_m(R/C)^\vee)
\] for all $i < d-1$. Therefore, $H^i_m(H^{d-1}_m(R/I)^\vee) = 0$ for all $i<d-1$ if and only if $H^i_m(H^{d-1}_m(R/C)^\vee) = 0$ for all $i <d-1$, as desired.  
\end{proof}

Recalling that Cohen--Macaulay is equivalent to sequentially Cohen--Macaulay and unmixed, it is not hard to see that Theorem \ref{thm:SCM} recovers the Cohen--Macaulayness implied by Corollary \ref{weakGVDimpliesgliggi} when all ideals appearing in  all the vertex decompositions throughout the induction are unmixed.

\begin{question}
Using Definition \ref{def:nonpureGVD} and its appropriate extension to an alternate definition of geometrically vertex decomposable, is every homogeneous geometrically vertex decomposable ideal sequentially Cohen--Macaulay?  Can we weaken the hypothesis in Theorem \ref{thm:SCM} that $R/N_{y,I}$ be Cohen--Macaulay to the hypothesis that it be merely sequentially Cohen--Macaulay?
\end{question}

\bibliographystyle{plain}
\bibliography{bibliography}

\end{document}